\documentclass[11pt,reqno]{amsart}
 
\usepackage[T1]{fontenc}
\usepackage[a4paper,top=4cm,bottom=4cm,inner=3cm,outer=3cm]{geometry}
\usepackage[pdftex]{graphicx,graphics} 
\usepackage{amsmath}
\usepackage{amssymb}
\usepackage{color,enumerate}
\usepackage{hyperref} 
\usepackage{epsfig,array} 

\usepackage{graphicx,psfrag,mathrsfs}
\usepackage{tikz}  
\hypersetup{
    linktoc=page,
   linkcolor=red,          
  citecolor=blue,        
    filecolor=blue,      
   urlcolor=cyan,
    colorlinks=true           
}

\usepackage[refpage]{nomencl} 
\newcommand{\lan}{{\mathcal L}}
\newcommand\var{\text{var}}

\newcommand\vl{\underline{v}}
\newcommand\wl{\underline{w}}

\RequirePackage{yhmath} 

\newcommand\N {{\mathbb N}}

\newtheorem{theorem}{Theorem}[section]
\newtheorem{corollary}{Corollary}

\newtheorem{lemma}[theorem]{Lemma}
\newtheorem{proposition}{Proposition}

\theoremstyle{definition}

\newtheorem{remark}{Remark}
\newtheorem*{notation}{Notation}

\title[shrinking target for random IFS] 
{shrinking target problem for random iterated function systems}

\author[Zhihui YUAN]{Zhihui Yuan}
\address{Zhihui Yuan, School of Mathematics and Statistics, Huazhong University of Science and Technology, 430074 Wuhan, P.R.China}

\email{yzhh@hust.edu.cn}

\begin{document}
	\maketitle
	
	
	%
	%
	%

\begin{abstract}
	We describe the shrinking target problem for random iterated function systems which semi-conjugate to a random subshifts of finite type. We get the Hausdorff dimension of the set based on shrinking target problems with given targets. The main idea is an extension of ubiquity theorem which plays an important role to get the lower bound of the dimension. Our method can be used to deal with the sets with respect to an more general targets and the sets based on the quantitative Poincar\'{e} recurrence properties. 
\end{abstract}
	
	\section{Introduction}
	
	This paper investigate the shrinking target problem for random iterated function systems. On a base probability space $(\Omega,\mathcal F,\mathbb{P})$ with a measure preserving ergodic transformation $\sigma$ on it,  we can define a family of random attractors generated by random iterated function systems. We will consider the Hausdorff dimension of the points whose orbit can be well approximated.
	
	Shrinking target problem is considering such a set:
	$$\{x\in X: T^nx\in A_n, \text{ for infinitely many } n\in\mathbb{N}\},$$
	where $\{A_n\}_{n\in\mathbb{N}}$ is a given (decreasing in some sense) sequence of subsets of the giving compact space $X$ and $T:X\to X$ is a continuous (normally expanding) map.  
	It was proposed in \cite{Hill1995,Hill1997} where they consider an expanding rational map on Riemann sphere acting on its Julia sets. There are two field to study such a set: in the sense of measure and dimension. See the results of measures in \cite{Boshernitzan1993,Chernov2001,Barreira2001,Maucourant2006,Galatolo2007,Fernandez2012,Kim2014} and  the results of dimension in \cite{Stratmann2002,Urbanski2002,Tan2011,Reeve2011,Bugeaud2014,Li2014,Seuret2015,Persson2017} for instances.
	
	Shrinking target problem for nonautonomous dynamical systems corresponding to Cantor series expansions has been considered in \cite{Fishman2015} (in the sense of dimension) and \cite{Sun2017}(in the sense of measure). There is a strong relationship between nonautonomous dynamical systems and random dynamical systems.
	In this paper, we study the dimension result of such problem for random iterated function systems, which can be seen as an extension of \cite{Fishman2015}. We derive the formula for Hausdorff dimension of the considering set after some reasonable assumptions. 

	The outline of the paper is as follows: Section~\ref{section:Setting and main results} develops background about random iterated function systems and presents our main results, namely  theorem~\ref{main} and corollary~\ref{W'} and~\ref{W''}. Section~\ref{section:Basic properties}  provides the basic properties that will be used in the proof of our results. Section~\ref{section:Upper bound} will give the upper bound while section~\ref{section:Lower bound} will give the lower bound by building a theorem of an extension of ubiquity theorem which can go back to \cite{Dodson1990,BS0,Beresnevich2006,BS2} to deal with limsup set. Section~\ref{section:corrollary} will return to explain the corollaries which can be seen as an extension of theorems~\ref{main}.

	\section{Setting and main result}\label{section:Setting and main results}
	
	\subsection{Random subshift and the pressure}\label{subsection: Random subshift}
	
	Denote by $\Sigma$ the symbolic space $({\mathbb Z^+})^\N$, and endow it with the standard ultrametric distance: for any $\underline u=u_0u_1\cdots$ and $\vl=v_0v_1\cdots$ in $\Sigma$, $d(u,v)=e^{-\inf\{n\in\mathbb N: \ u_n\neq v_n\}}$, with the convention $\inf (\emptyset)=+\infty$. Let $(\Omega,\mathcal{F},\mathbb{P})$ be a complete probability space and $\sigma$  a $\mathbb{P}$-preserving  ergodic map. The product space $\Omega\times \Sigma$ is endowed with the $\sigma$-field $\mathcal F\otimes \mathcal B(\Sigma)$, where $\mathcal B(\Sigma)$ stands for the Borel $\sigma$-field of $\Sigma$.
	
	Let $l$ be a $\mathbb{Z}^+$ valued random variable such that
	$$
	\int\log (l)\, \mathrm{d}\mathbb{P}<\infty \text{ and }  \mathbb{P}(\{\omega\in\Omega: l(\omega)\geq 2\})>0.
	$$
	Let $A=\{A(\omega)=(A_{r,s}(\omega)):\omega\in\Omega\}$ be  a random transition matrix such that $A(\omega)$ is a $l(\omega)\times l(\sigma\omega)$-matrix with entries 0 or 1. We  suppose that the map $\omega\mapsto A_{r,s} (\omega)$ is measurable for all $(r,s)\in \mathbb{Z}^+\times \mathbb{Z}^+$ and each $A(\omega)$  has at least one non-zero entry in each row and each column.
	Let $$\Sigma_{\omega}=\{\vl=v_0v_1\cdots;1\leq v_k\leq l(\sigma^{k}(\omega))\text{ and } A_{v_k,v_{k+1}}(\sigma^{k}(\omega))=1\ \text{for all } k\in\mathbb{N} \},$$
	and $F_{\omega}:\Sigma_{\omega}\rightarrow\Sigma_{\sigma\omega}$ be the left shift $(F_{\omega}\vl)_i=v_{i+1}$ for any $\vl=v_0v_1\cdots\in \Sigma_{\omega}$. Define $\Sigma_{\Omega}=\{(\omega,\vl):\omega\in\Omega, \vl\in\Sigma_\omega\}$.
	The space $\Sigma_\Omega$ is endowed with the $\sigma$-field  obtained as the trace of $\mathcal F\otimes \mathcal B(\Sigma)$. 
	Define the map $F:\Sigma_{\Omega}\to \Sigma_{\Omega}$  as $F((\omega,\vl))=(\sigma\omega,F_{\omega}\vl)$. The corresponding family $\tilde{F}=\{F_{\omega}:\omega\in \Omega\}$ is called a random subshift. We assume that this  random subshift  is topologically mixing, i.e.  there exists a $\mathbb{Z}^+$-valued random variable $M$ on $(\Omega,\mathcal{F},\mathbb{P})$ such that for $\mathbb{P}$-almost every (a.e.) $\omega$,
	$
	A(\omega)A(\sigma\omega)\cdots A(\sigma^{M(\omega)-1}\omega)$ is positive (i.e. each entry of the matrix is positive).

	For each $n\in\mathbb{N}$, define $\Sigma_{\omega,n}$ as the set of words $v=v_0v_1\cdots v_{n-1}$ of length $n$, i.e. such that
	$1\leq v_k\leq l(\sigma^{k}(\omega))$ for all  $0\leq k\leq n-1$
	and $A_{v_k,v_{k+1}}(\sigma^{k}(\omega))=1$ for all $ 0\leq k\leq n-2$.
	Define $\Sigma_{\omega,*}=\cup_{n\in\mathbb{N}}\Sigma_{\omega,n}$.
	For $v=v_0v_1\cdots v_{n-1}\in\Sigma_{\omega,n}$,  we write $|v|$ for the length $n$ of $v$, and we define the cylinder $[v]_{\omega}$ as
	$[v]_{\omega}:=\{\wl\in\Sigma_\omega:w_i=v_i\text{ for }i=0,\dots,n-1\}$  and $v^\star=v_0v_1\cdots v_{n-2}\in\Sigma_{\omega,n-1}$.  
	
	For any $s\in \Sigma_{\omega,1}$, $p\geq M(\omega)$ and $s'\in \Sigma_{\sigma^{p+1}\omega,1}$, there is  at least one word $v(s,s')\in \Sigma_{\sigma\omega,p-1}$ such that $sv(s,s')s'\in \Sigma_{\omega,p+1}$. We fix such a $v(s,s')$ and denote the word $sv(s,s')s'$ by $s\ast s'$. Similarly, for any $w=w_0w_1\cdots w_{n-1}\in\Sigma_{\omega,n}$ and $w'=w'_0w'_1\cdots w'_{m-1}\in\Sigma_{\sigma^{n+p}\omega,m}$ with $n,p,m\in\mathbb{N}$ and $p\geq M(\sigma^{n-1}\omega)$, we  fix $v(w_{n-1},w'_{0})\in \Sigma_{\sigma^n\omega, p}$ (a word depending on  $w_{n-1}$ and $w'_{0}$ only) so that  $w\ast w':=w_0w_1\cdots w_{n-1}v(w_{n-1},w'_{0})w'_0w'_1\cdots w'_{m-1}\in\Sigma_{\omega,n+m+p-1}$.
	
	We say that a measurable function $\Upsilon:\Sigma_{\Omega}\to \mathbb{R}$ is in $\mathbb{L}^1(\Omega,C(\Sigma))$ if
	\begin{enumerate}
		\item \begin{equation*}\label{int}
		C_{\Upsilon}=:\int_{\Omega} \|\Upsilon(\omega)\|_{\infty}\, \mathrm{d}\mathbb{P}(\omega)<\infty,
		\end{equation*}
		where
		$
		\|\Upsilon(\omega)\|_{\infty}=:\sup_{\underline{v}\in \Sigma}|\Upsilon(\omega,\vl)|$,
		\item 
		for any $n\in\mathbb{N}$, there exists a random variable $\var_n(\Upsilon,\cdot)$ such that 
		$$\sup\{|\Upsilon(\omega,\vl)-\Upsilon(\omega,\wl)| : v_i = w_i, \forall i < n\}\leq \var_n(\Upsilon,\omega).$$
		Furthermore, for $\mathbb{P}$-a.e. $\omega\in\Omega$, $\var_n(\Upsilon,\omega) \to 0$ as $n\to \infty$.
	\end{enumerate}
	
	Now, if $\Upsilon\in \mathbb{L}^1(\Omega,C(\Sigma))$, due to Kingsman's subadditive ergodic theorem,
	\begin{equation}\label{def Pressure}
	P(\Upsilon)=\lim_{n\to\infty}\frac{1}{n}\log \sum_{v\in\Sigma_{\omega,n}}\sup_{\underline v\in[v]_\omega}\exp\left(S_n\Upsilon(\omega,\underline v)\right )
	\end{equation}
	exists for $\mathbb{P}$-a.e.  $\omega$ and does not depend on $\omega$, where $S_n\Upsilon(\omega,\underline v)=\sum_{i=0}^{n-1}\Upsilon(F^i(\omega,\underline v))$.  This limit is called {\it topological pressure of $\Upsilon$}.
	
	\subsection{A model of random iterated function systems}\label{subsection:A model of random iterated function systems}
	
	We present the model of random iterated function systems.
	Fixed any nonempty compact set $U\subset \mathbb{R}^d$ such that $U=\overline{{\rm Int}(U)}$. We assume that for any $s\in\mathbb{N}$, for any $\omega\in\Omega$, there exists a homeomorphism map $g^s_{\omega}$ from $U_{\sigma\omega}=U$ to $g^s_{\omega}(U)\subset U$. Denote $g_{\omega}^{v}=:g_{\omega}^{v_0}\circ g_{\sigma\omega}^{v_1}\circ\cdots\circ g_{\sigma^{n-1}\omega}^{v_{n-1}}$ for $v=v_0v_1\cdots v_{n-1}\in\Sigma_{\omega,n}$.
	We can define the following sets:
	\begin{eqnarray*}
		U_{\omega}^v&=&g_{\omega}^{v}(U),\ \forall v\in\Sigma_{\omega,n},\\
		X_{\omega}&=&\bigcap_{n\geq 1}\bigcup_{v\in\Sigma_{\omega,n}}U_{\omega}^v\quad {\rm and }\quad X_{\omega}^v=X_{\omega}\cap U_{\omega}^v,\\
		X_{\Omega}&=&\{(\omega,x):\omega\in\Omega, x\in X_{\omega}\},
	\end{eqnarray*}
	Denote $T_{\omega}^s=(g^s_{\omega})^{-1}: U_{\omega}^s\to U_{\sigma\omega}=U$ as its inverse function and $T_{\omega}^{v}=:T_{\sigma^{n-1}\omega}^{v_{n-1}}\circ T_{\sigma^{n-2}\omega}^{v_{n-2}}\circ\cdots\circ T_{\omega}^{v_{0}}$ for $v=v_0v_1\cdots v_{n-1}\in\Sigma_{\omega,n}$. we need to point out that $g_{\omega}^{v}$ is a map from $U$ to $U_{\omega}^{v}$ and $T_{\omega}^{v}: U_{\omega}^{v} \to U$ is the inverse map of $g_{\omega}^{v}$.

	It is easy to check that $g^{v_0v_1\cdots v_{n-1}}_{\omega}(X_{\sigma^{n}\omega}^{v_{n}})=X_{\omega}^v$ for any $v=v_0v_1\cdots v_{n-1}v_{n}\in\Sigma_{\omega,n+1}$ with $n\in\mathbb{N}$.
	
	We say that a function $\widetilde \psi: \widetilde{U}_{\Omega}=\{(\omega,s,x):\omega\in \Omega,1\leq s\leq l(\omega),x\in U\} \to \mathbb{R}$ is in $\mathbb{L}^1(\Omega,\widetilde{C}(U))$ if
	\begin{enumerate}
		\item for any $1\leq s\leq l(\omega), x\in U$, the map $\omega\mapsto \widetilde \psi(\omega,s,x)$ is measurable, 
		\item 
		$\int_{\Omega}\|\widetilde\psi(\omega)\|_{\infty}\mathrm{d}\mathbb{P}(\omega)<\infty$,
		where
		$\|\widetilde\psi(\omega)\|_{\infty}:=\sup_{1\leq s\leq l(\omega)}\sup_{x\in U_{\omega}^{s}}|\widetilde\psi(\omega,s,x)|$,
		
		\item for any $\varepsilon>0$, there exists a random variable $\var(\widetilde\psi,\cdot,\varepsilon)$ such that 
		$$\sup_{1\leq s\leq l(\omega)}\sup_{x,y\in U,|x-y|\leq \varepsilon }|\widetilde\psi(\omega,s,x)-\widetilde\psi(\omega,s,y)|\leq \var(\widetilde\psi,\omega,\varepsilon).$$
		Furthermore, for $\mathbb{P}$-a.e. $\omega\in\Omega$, $\var(\widetilde\psi,\omega,\varepsilon) \to 0$ as $\varepsilon\to 0$.
	\end{enumerate}
	The following assumptions will be needed throughout the paper.
	\begin{enumerate}
		\item ${\rm Int}(U_{\omega}^{s_1})\cap {\rm Int}(U_{\omega}^{s_2})=\emptyset$ for all $\omega\in\Omega$ and for all $1\leq s_1,s_2\leq l(\omega)$ with $s_1\neq s_2$.
		\item there exists $\psi\in\mathbb{L}^1(\Omega,\widetilde{C}(U))$ such that 
		\begin{equation}\label{cpsi}
		c_{\psi}:=-\int_{\Omega}\sup_{1\leq s\leq l(\omega)}\sup_{x\in U}\psi(\omega,s,x)\mathrm{d}\mathbb{P}(\omega)>0
		\end{equation}
		and
		\begin{equation}\label{distortion}
		\exp(-\var(\psi,\omega,\|x-y\|))\leq\frac{\|g_{\omega}^s(x)-g_{\omega}^s(y)\|}{\|x-y\|\exp(\psi(\omega,s,x))}\leq \exp(\var(\psi,\omega,\|x-y\|)).
		\end{equation}
	\end{enumerate}

	Under the assumptions, fixed $x_0\in U$, there is $\mathbb{P}$-almost surely a natural projection $\pi_\omega: \Sigma_{\omega}\rightarrow X_{\omega}$ defined as
	$$
	\pi_\omega(\underline{v})= \lim_{n\rightarrow\infty}g_{\omega}^{v_0}\circ g_{\sigma\omega}^{v_1}\circ\cdots\circ g_{\sigma^{n-1}\omega}^{v_{n-1}}(x_0).$$
	Noticing that this mapping may not be injective, but it should be surjective.
	\subsection{Shrinking target problem and result}
	Let $\phi\in\mathbb{L}^1(\Omega,\widetilde{C}(U))$ such that 
	\begin{equation}\label{cphi}
	c_{\phi}:=-\int_{\Omega}\sup_{1\leq s\leq l(\omega)}\sup_{x\in U}\phi(\omega,s,x)\mathrm{d}\mathbb{P}(\omega)> 0.
	\end{equation}
	For any $\omega\in\Omega$, fix any point $z_\omega\in X_\omega$. Now we consider the points in $X_{\omega}$ whose orbits are well approximated by the sequence $\{z_{\sigma^n\omega}\}_{n\in\mathbb{N}}$ with rate corresponding to $\phi$, that is 
	$$W(\phi,\omega)=\left\{x\in X_{\omega}:\begin{array}{l}
	x\in U_{\omega}^v \text{ and }\|T_{\omega}^{v}x-z_{\sigma^{|v|}\omega}\|\leq \exp(S_{|v|}\phi(\omega,x))\\
	\text{ for infinitely many } v\in\Sigma_{\omega,*}
	\end{array}
	\right\}.$$
	where for any $n\in\mathbb{N}$, for any $v=v_0v_1\cdots v_{n-1}\in \Sigma_{\omega,n}$, for $x\in U_{\omega}^v$, $S_n\phi(\omega,x):=\phi(\omega,x)+\sum_{i=1}^{n-1}\phi(\sigma^i\omega,T_{\omega}^{v_0\cdots v_{i-1}}x)$.
	The main aim of the shrinking target problem is to consider the dimension and the measure of the set $W(\phi,\omega)$. In this paper we will fix our effort on the Hausdorff dimension of it.

	Define
	$$
	\Psi(\omega,\vl)=\psi(\omega,v_0,\pi(\vl))\text{ and }\Phi(\omega,\vl)=\phi(\omega,v_0,\pi(\vl)),
	$$
	for any $\omega\in\Omega$ and any $\vl=v_0v_1\cdots\in\Sigma_{\omega}$.
	By construction, $\Psi$ and $\Phi$ are in $\mathbb{L}^1(\Omega,C(\Sigma))$.
	\begin{theorem}\label{main}
		{\rm	Under our assumption, for $\mathbb{P}$-a.e. $\omega\in\Omega$, the Hausdorff dimension of the set $W(\phi,\omega)$ is the unique root $q_0$ of the equation $P(q(\Psi+\Phi))=0$, where the pressure function $P$ is defined in \eqref{def Pressure}. }
	\end{theorem}
	
	In theorem \ref{main}, for a fixed $\omega\in\Omega$, the targets are $\{z_{\sigma^n\omega}\}_{n\in\mathbb{N}}$ which are just depends on $n$. Using a similar method of the proof of the theorem, we can deal with general targets $\{z_{\sigma^{|v|}\omega}^v\}_{v\in\Sigma_{\sigma^{|v|}\omega}}$ which are depends on $v$. It is described in the following corollary which will be explained in section~\ref{section:corrollary}. 
	\begin{corollary}\label{W'}
		{\rm For any $\omega\in\Omega$, for any $n\in\mathbb{Z}^{+}$, for any $v=v_0v_1\cdots v_{n-1}\in\Sigma_{\omega,n}$, fix $z_{\sigma^n\omega}^v\in X_{\sigma^n\omega}$.
			We make the following assumption: there exists a $\mathbb{Z}^+$-valued r.v. $M'$ on $(\Omega,\mathcal{F},\mathbb{P})$ such that
			\begin{equation}\label{Ass W'-1}
			\mathbb{P}\left(\left\{\omega\in\Omega: \begin{array}{l}
			\forall n\in\mathbb{Z}^{+},\forall v\in\Sigma_{\omega,n},\text{ there exist } k\in\mathbb{N},\\
			k\leq M'(\sigma^{n-1}\omega)\text{ and } v'\in\Sigma_{\sigma^n\omega,k}\text{ such that: }\\
			vv'\in\Sigma_{\omega,n+k} \text{ and } g_{\omega}^{vv'}(z_{\sigma^{n+k}\omega}^{vv'})\in X_{\omega}^{vv'}
			\end{array}\right\}\right)=1
			\end{equation}
			Now we consider the following set
			$$W'(\phi,\omega)=\left\{x\in X_{\omega}:\begin{array}{l}
			x\in U_{\omega}^v \text{ and }\|T_{\omega}^{v}x-z_{\sigma^{|v|}\omega}^v\|\leq \exp(S_{|v|}\phi(\omega,x))\\
			\text{ for infinitely many } v\in\Sigma_{\omega,*}
			\end{array}
			\right\}.$$
			The Hausdorff dimension of $W'(\phi,\omega)$ is equal to the unique root $q_0$ of the equation $P(q(\Psi+\Phi))=0$.
		}
	\end{corollary}

    Here, the equation \eqref{Ass W'-1} means that the target can be hit. In the fullshifts situation for each point $z_{\sigma^n\omega} \in X_{\sigma^n\omega}$, it can be hit for some $x\in X_{\omega}^{v}$ for each $v\in\Sigma_{\omega,n}$ which means $T_{\omega}^v x=z_{\sigma^n\omega}$ since $T_{\omega}^{v}=X_{\sigma^n\omega}$. We make a more general condition that we do not need to wait too long to get a point that can be hit (that is $g_{\omega}^{vv'}(z_{\sigma^{n+k}\omega}^{vv'})\in X_{\omega}^{vv'}$). In fact, here the random variable $M'$ play the same role as $M$. Instead of the equation~\eqref{Ass W'-1} in the assumption, we can also use the following condition: For $\mathbb{P}$-a.e. $\omega\in\Omega$, there exist a sequence $\{\gamma_n\}_{n\in\mathbb{N}}$ of positive numbers decreasing to 0 such that
    for all $n\in\mathbb{Z}^{+}$, for all $v\in\Sigma_{\omega,n}$, there exist $k\in\mathbb{N},k\leq \gamma_n\text{ and } v'\in\Sigma_{\sigma^n\omega,k}$
    satisfying 
    	$vv'\in\Sigma_{\omega,n+k}$ and $g_{\omega}^{vv'}(z_{\sigma^{n+k}\omega}^{vv'})\in X_{\omega}^{vv'}$. 
	\begin{corollary}\label{W''}
		{\rm 
			We make the following assumption: there exists a $\mathbb{Z}^+$-valued r.v. $M''$ on $(\Omega,\mathcal{F},\mathbb{P})$ such that 
			\begin{equation}\label{Ass W''-1}
			\mathbb{P}\left(\left\{\omega\in\Omega: \begin{array}{l}
			\forall n\in\mathbb{Z}^{+},\forall v\in\Sigma_{\omega,n},\text{ there exist } k\in\mathbb{N},\\
			k\leq M''(\sigma^{n-1}\omega)\text{ and } v'\in\Sigma_{\sigma^n\omega,k}\text{ such that: }\\
			vv'\in\Sigma_{\omega,n+k} \text{ and } T_{\omega}^{vv'}x_{\omega}^{vv'}=x_{\omega}^{vv'}\in X_{\omega}^{vv'}
			\end{array}\right\}\right)=1.
			\end{equation}
			
			Consider the following set 
			$$W''(\phi,\omega)=\left\{x\in X_{\omega}:\begin{array}{l}
			x\in U_{\omega}^v \text{ and }\|T_{\omega}^{v}x-x\|\leq \exp(S_{|v|}\phi(\omega,x))\\
			\text{ for infinitely many } v\in\Sigma_{\omega,*}
			\end{array}
			\right\}.$$	
			The Hausdorff dimension of $W''(\phi,\omega)$ is equal to the unique root $q_0$ of the equation $P(q(\Psi+\Phi))=0$.
		}
	\end{corollary}
	
	$W''(\phi,\omega)$ can be seen as the points $x$ whose orbits come back closer and closer to $x$ at a rate $\phi$ possibly depending on $x$ which is also called "quantitative Poincar\'{e} recurrence property".  Equation \eqref{Ass W''-1} in the assumption of corollary~\ref{W''} is not surprising. First, for $\mathbb{P}$-a.e. $\omega\in\Omega$, for any $v\in\Sigma_{\omega,n}$ with $n$ large enough, the fixed $x_{\omega}^v$ of $T_{\omega}^{v}$ (also $g_{\omega}^{v}$) exists. In fact we just need to notice \eqref{cpsi} and \eqref{distortion}, using Birkhoff ergodic theorem, we know that the map $g_{\omega}^{v}$ is a contraction, the contraction mapping principle ensures the existence of the fixed point.
	Second, we want to deal with $\|T_{\omega}^{v}x-x\|\leq\exp(S_{|v|}\phi(\omega,x))$ which can be seen as they are very near to the fixed point of $T_{\omega}^{v}$. Third, since we want to deal with the fixed point of $T_{\omega}^{v}:X_{\omega}^v\to X_{\sigma^{|v|}\omega}$, it is reasonable to assume the fixed point is in $X_{\omega}^v\cap X_{\sigma^{|v|}\omega}$. At last, in the deterministic situation with full coding, it is easy to see that the fixed point is inside the attractor.
	
	\begin{remark}
		\begin{itemize}
			\item In fact if we have the separation condition, that is $U_{\omega}^{s_1}\cap U_{\omega}^{s_1}=\emptyset$ for any $1\leq s_1,s_2\leq l(\omega)$ with $s_1\neq s_2$, we can define $T_{\omega}:\cup_{1\leq s\leq l(\omega)}U_{\omega}^{s}\to U$ such that $T_{\omega}|_{U_{\omega}^{s}}=T_{\omega}^s$. In our setting we know their inners are pairwise disjoint, but they may intersect on their boundary where there is a trouble to define $T_{\omega}$. 
			
			\item In fact, both \eqref{cpsi} and \eqref{cphi} can be replaced as: there exists $n\in\mathbb{N}$ such that
			\begin{equation*}\label{cpsi'}
			-\int_{\Omega}\sup_{v\in\Sigma_{\omega,n}}\sup_{x\in U_{\omega}^v}S_n\widetilde\psi(\omega,x)\mathrm{d}\mathbb{P}(\omega)> 0,
			\end{equation*} 
			where $\widetilde\psi\in\{\psi,\phi\}$
		\end{itemize}
	\end{remark}
	
	Let us make some comments on our setting and result:
	\begin{itemize}
		\item 
		In many works on the shrinking target problems (for example \cite{Hill1995,Hill1997,Stratmann2002}), there will be conformal situation which fails in our  case. This takes a big trouble to control measures of balls. Since we are dealing with higher dimension, it is not easy to use the trick in \cite{Yuan2017DCDS,Yuan2017} where we give a good control of neighbor cylinders.  Without  the method in \cite{Urbanski2002,Reeve2011}, we overcome such a trouble (see the proof of lemma \ref{lemma: Control zeta Bxr}) by using the cylinders with radius about $r$ to cover a ball $B(x,r)$, and noticing the number of them is not so big since they are inside $B(x,2r)$.  
		\item As was noted, the maps $g_{\omega}^s$ with $1\leq s\leq l(\omega)$ may not be  contraction, but they are contraction in average which means they are contractions if we look for a long time. 
		\item We now deal with subshift situation which have not been considered before.  
		\item As in \cite{Stratmann2002,Urbanski2002,Tan2011,Reeve2011,Bugeaud2014,Li2014,Seuret2015,Persson2017}, they are dealing the deterministic situation. In our case, we deal with the random iterated function systems with respect to a random subshift. The randomness makes many more models are fall in our situation, but more difficult to treat. Especially, the random variable $M$ in the topologically mixing property is no longer constant. Furthermore, the random iterated function systems in our situation is of $C^1$ which makes the pressure function no longer differentiable, so we need to distribute the measure by step in the manner of approximation.   Here we define a sequence of good sets $(\Omega_i)_{i\in\mathbb{N}}$ such that each point in them has good properties we want. Also, the ubiquity theorem (see \cite{Dodson1990,BS0,Beresnevich2006,BS2}) has been developed in our situation. 
		\item In \cite{Fishman2015}, if we add some restriction on $Q=(q_i)_{i\in\mathbb{N}}$, such as $limsup_{n\to \infty}\frac{\sum_{i=0}^{n+1}q_i}{n}< +\infty$, the result can be seen as a good sample in our theorem.
		\item We have consider two classes of shrinking target problems. The one is for given targets (see theorem~\ref{main} and corollary~\ref{W'}), the other is with respect to quantitative Poincar\'{e} recurrence properties (see corollary~\ref{W''}). The sequence $(z_{\sigma^n\omega})_{n\in\mathbb{N}}$ in $W(\phi,\omega)$ of theorem~\ref{main} can be see as $(z_n)_{n\in\mathbb{N}}$ which are always to be $z$ in  deterministic system. Furthermore, $W'(\phi,\omega)$ is an extension of $W(\phi,\omega)$, and it also give a good description of $W''(\phi,\omega)$ as a special case when the fixed points are chosen.	
	\end{itemize}

	\section{Basic properties}\label{section:Basic properties}
	
	Now we will introduce random weak Gibbs measures and the related properties which have been prepared in the previous works. 
	
	For $\Upsilon\in\mathbb{L}^1(\Omega,C(\Sigma))$, the associated Ruelle-Perron-Frobenius  operator $\lan_{\Upsilon}^\omega: C^0(\Sigma_{\omega})\to C^0(\Sigma_{\sigma\omega})$ is defined as
	$$\lan_{\Upsilon}^\omega h(\vl)=\sum_{F_{\omega}\wl=\vl}\exp(\Upsilon(\omega,\wl))h(\wl),\ \ \forall\ \vl\in \Sigma_{\sigma\omega}.$$
	\begin{proposition}\label{eigen}\cite{Kifer1,MSU}
		Removing from $\Omega$ a set of $\mathbb{P}$-probability 0 if necessary, for all  $\omega\in \Omega$ there exists $\lambda(\omega)=\lambda^{\Upsilon}(\omega)>0$ and  a probability measure $\widetilde \mu_{\omega}=\widetilde\mu^\Upsilon_\omega$ on $\Sigma_{\omega}$ such that $(\mathcal{L}_{\Upsilon}^\omega)^*\widetilde \mu_{\sigma\omega}=\lambda(\omega)\widetilde \mu_{\omega}$.
	\end{proposition}
	We call the  family $\{\widetilde \mu_{\omega}:\omega\in\Omega\}$ a random weak Gibbs measure on $\{\Sigma_\omega:\omega\in\Omega\}$ associated with $\Upsilon$. If we want, we can use $\{\widetilde \mu^{\Upsilon}_{\omega}:\omega\in\Omega\}$ to declare that it is with respect to $\Upsilon\in\mathbb{L}^1(\Omega,C(\Sigma))$. 
	
	Let $u=\{u_{n,\omega}\}$ be an extension and $\Upsilon\in \mathbb{L}^1(\Omega,C(\Sigma))$. Then for $(n,\omega)\in \mathbb{N}\times\Omega$
	\begin{equation*}
	Z_{u,n}(\Upsilon,\omega):=\sum_{v\in\Sigma_{\omega,n}}\exp\big (S_n\Upsilon(\omega,u_{n,\omega}(v))\big)
	\end{equation*}
	is called $n$-th partition function of $\Upsilon$ in $\omega$ with respect to $u$.
	
	Due to the assumption $\log (l)\in\mathbb{L}^1(\Omega,\mathbb{P})$, using the same method as in \cite{Gundlach,KL}, it is easy to prove the following lemma.
	
	\begin{lemma}\label{converge Pressure}\cite{Yuan2017DCDS}
		Let $u$ be any extension and $\Phi\in \mathbb{L}^1(\Omega,C(\Sigma))$.
		
		Then $\lim_{n\rightarrow\infty}\frac{1}{n}\log Z_{u,n}(\Phi,\omega)=P(\Phi)$ for $\mathbb{P}$-almost every $\omega\in\Omega$. This limit is independent of $u$.
	\end{lemma}

	Using a standard approach, it can be easily proven that for $\mathbb{P}$-almost every $\omega\in \Omega$, the Bowen-Ruelle formula holds, i.e. $\dim_H X_{\omega}=t_0$ where $t_0$ is the unique root of the equation $P(t\Psi)=0$.
	
	Noting the assumptions in subsection~\ref{subsection:A model of random iterated function systems}, especially $\psi\in\mathbb{L}^1(\Omega,\widetilde{C}(U))$ and \eqref{distortion}, using the same method in \cite[proposition 3]{Yuan2017DCDS}, we can proof:
	\begin{proposition}\label{Length and measure}{\rm
			For $\mathbb{P}$-almost every  $\omega\in \Omega$, there are sequences $(\epsilon(\psi,\omega,n))_{n\in\mathbb{N}}$ (also denote as $(\epsilon(\Psi,\omega,n))_{n\in\mathbb{N}}$) and $(\epsilon(\Upsilon,\omega,n))_{n\in\mathbb{N}}$  of positive numbers, decreasing to 0 as $n\to +\infty$, such that for all $n\in \mathbb{N}$, for all $v=v_0v_1\dots v_{n}\in\Sigma_{\omega,n}$, we have :
			\begin{enumerate}
				\item  For all $z\in U_{\omega}^v,$
				$$
				|U_{\omega}^v|\leq \exp(S_n\psi(\omega,z)+n\epsilon(\psi,\omega,n)),$$
				furthermore, there exists a ball $B$ with radius $\exp(S_n\psi(\omega,z)-n\epsilon(\psi,\omega,n))$ such that $B\subset U_{\omega}^v$. 
				
				Hence for all $\vl\in [v]_{\omega}$,
				$$|U_{\omega}^v|\leq \exp(S_n\Psi(\omega,\vl)+n\epsilon(\Psi,\omega,n)),$$
				and  there exists a ball $B$ with radius $\exp(S_n\Psi(\omega,\vl)-n\epsilon(\Psi,\omega,n))$ such that $B\subset U_{\omega}^v$.
				\item The measure $\widetilde{\mu}^{\Upsilon}_{\omega}$ (see details in proposition~\ref{eigen}) exists, and for all $\vl\in[v]_{\omega}$, we have 
				$$
				\exp(-n\epsilon(\Upsilon,\omega,n))\leq \frac{\widetilde\mu^{\Upsilon}_{\omega}([v]_{\omega})}{\exp(S_n\Upsilon(\omega,\vl)-nP(\Upsilon)}\leq \exp(n\epsilon(\Upsilon,\omega,n)).$$
		\end{enumerate}}
	\end{proposition}
	
	\begin{remark}
		{\rm 
			\begin{itemize}
				\item By Maker's ergodic theorem from \cite{Maker1940}, we can get that for $\Upsilon\in\mathbb{L}^1(\Omega,C(\Sigma))$, for $\mathbb{P}$-a.e. $\omega\in\omega$ we have $\lim_{n\rightarrow\infty}\frac{\sum_{i=0}^{n-1}\var_{n-i}(\Upsilon,\sigma^{i}\omega)}{n}= 0$. 
				Without any difficulty, we can ask 
				\begin{equation*}\label{Control V-n}
				\frac{\sum_{i=0}^{n-1}\var_{n-i}(\Upsilon,\sigma^{i}\omega)}{n}\leq \epsilon(\Upsilon,\omega,n)
				\end{equation*}
				in proposition~\ref{Length and measure}
				\item Using item 3 in the definition of $\mathbb{L}^1(\Omega,\widetilde{C}(U))$, also by Maker's ergodic theorem, we can ask that for any $n\in\mathbb{Z}^+$, for any $v\in\Sigma_{\omega,n}$ and any $x,y$ in $U_{\omega}^v$, we have 
				\begin{equation}\label{Control V-n-psi}
				|S_n\phi(\omega,x)-S_n\phi(\omega,y)|\leq n\epsilon(\Phi,\omega,n)
				\end{equation}
				for any $x,y$ in $U_{\omega}^v$.
				\item  From item 1 of proposition~\ref{Length and measure}, we can easily get that for $\mathbb{P}$-almost every  $\omega\in \Omega$, there is a sequence $(\epsilon'(\psi,\omega,n))_{n\in\mathbb{N}}$ of positive numbers, decreasing to 0 as $n\to +\infty$, such that for all $n\in \mathbb{N}$, for any $v\in\Sigma_{\omega,n}$, $U_{\omega}^{v}$ contains a ball with radius $|U_{\omega}^{v}|^{1+\epsilon'(\Psi,\omega,n)}$. we do not distinguish between $\epsilon'(\Psi,\omega,n)$ and  $\epsilon(\Psi,\omega,n)$ for the convenient of writing.
			\end{itemize}
			
		}
	\end{remark}

	We now start a series of estimations that will be useful later. These estimations can be seen from~\cite{Yuan2017DCDS,Yuan2017}. 
	At first, choose $ \widetilde{M}\in\mathbb{N}$ large enough such that
	\begin{equation}\label{M big}
	\mathbb{P}(\{\omega\in\Omega: M(\omega)\leq \widetilde{M}\})>7/8.
	\end{equation}
	Second, Birkhoff ergodic theorem and Egorov's theorem yields there exist $C>0$ and a measurable subspace $\widetilde\Omega\subset \{\omega\in\Omega: M(\omega)\leq  \widetilde{M}\}$ such that $\mathbb{P}(\widetilde \Omega)>6/7$ and for all $\omega\in\widetilde\Omega$, $n\in\mathbb{N}$ and  for $\Upsilon\in\{\Phi,\Psi\}$,  one has
	\begin{equation}\label{Control upsilon upper}
	\max\left(\frac{1}{n}S_n\|\Upsilon(\omega)\|_{\infty},\ \frac{1}{n}S_n\|\Upsilon(\sigma^{-n+1}\omega)\|_{\infty}\right)\leq  C. 	
	\end{equation}
	Furthermore, there exist $c>0$ small enough and  $\mathcal N\in\mathbb{N}$ with $\mathcal{N}> \widetilde M$ large enough such that for any $n\geq N$ one has 
	\begin{equation}\label{Control Psi lower}
	S_n\Psi(\omega,\vl)\leq -cn
	\end{equation}

	Third, we say that a function $\Upsilon: \Sigma_{\Omega}\to \mathbb{R}$ is a random H\"{o}lder continuous potential if 
	\begin{align*}
	&(1)\ \Upsilon\in \mathbb{L}^1(\Omega,C(\Sigma));\\ 
	&(2)\  \text{$\exists\ \kappa\in (0,1]$:\ }
	\text{var}_n\Upsilon(\omega)\leq K_{\Upsilon}(\omega)e^{-\kappa n}, \text{ with $\int\log K_\Upsilon (\omega)\mathrm{d}\mathbb{P}(\omega)<\infty$}.
	\end{align*}
	For any $\{\varepsilon_i\}_{i\in\mathbb{N}}$ be a sequence of positive numbers, let $\{\Psi_i\}_{i\in\mathbb{N}},\{\Phi_i\}_{i\in\mathbb{N}}$ be two sequences of random H\"{o}lder potentials such that
	\begin{equation}\label{var Psiphi-i}
	\int_{\Omega} \max\{\|(\Psi-\Psi_i)(\omega)\|_{\infty},\|(\Phi-\Phi_i)(\omega)\|_{\infty}\}\, \mathrm{d}\mathbb{P}(\omega)<\varepsilon^3_i.
	\end{equation}
	
	For each $i\in\mathbb{N}$, using Birkhoff ergodic theorem and Egorov's theorem, there exists a measurable set $\Omega(i)\subset\widetilde \Omega$ and $\mathsf{N}_i\in\mathbb{N}$ such that 
	\begin{itemize}
		\item $\mathbb{P}(\Omega(i))>3/4$,
		\item for any $n\geq \mathsf{N}_i$,
		\begin{equation}\label{Control averagepsi upsilon}
		\left|S_n\|(\Psi-\Psi_i)(\sigma^{\mathcal{N}}\omega)\|_{\infty}-n\int_{\Omega} \|(\Psi-\Psi_i)(\omega)\|_{\infty}\ d\mathbb{P}\right| \leq n\varepsilon_i^3,
		\end{equation}
		and
		\begin{equation}\label{Control averagephi upsilon}
		\left|S_n\|(\Phi-\Phi_i)(\sigma^{\mathcal{N}}\omega)\|_{\infty}-n\int_{\Omega} \|(\Phi-\Phi_i)(\omega)\|_{\infty}\ d\mathbb{P}\right| \leq n\varepsilon_i^3,
		\end{equation}
		\item by using proposition \ref{Length and measure}, we can ask that for each $\omega\in\Omega(i)$, the measure $\widetilde\mu_{\sigma^{\mathcal{N}}\omega}^{i}=:\widetilde\mu_{\sigma^{\mathcal{N}}\omega}^{q_i(\Psi_i+\Phi_i)}$ is well defined such that for any $n\in\mathbb{N}$, for any $v\in\Sigma_{\sigma^{\mathcal{N}}\omega,n}$
		\begin{multline*}\label{measure u i}
		\exp(n\epsilon(q_i(\Psi_i+\Phi_i),\sigma^{\mathcal{N}}\omega,n))\\
		\leq\frac{\widetilde\mu^{i}_{\omega}([v]_{\omega})}{\exp(q_iS_n(\Psi_i+\Phi_i)(\sigma^{\mathcal{N}}\omega,\vl)}\\
		\leq \exp(n\epsilon(q_i(\Psi_i+\Phi_i),\sigma^{\mathcal{N}}\omega,n)),
		\end{multline*}
		and
		\begin{equation*}\label{Control epsilon-Lambda-n}    
		\epsilon(q_i(\Psi_i+\Phi_i),\sigma^{\mathcal{N}}\omega,n)\leq \varepsilon_i^3, \text{ for $n\geq \mathsf{N}_i$.}
		\end{equation*} 
	\end{itemize}

	Let
	\begin{equation*}
	F_{i,\beta,n}(\sigma^{\mathcal{N}}\omega,{\varepsilon})\\
	=:\left\{\vl\in \Sigma_{\sigma^{\mathcal{N}}\omega}:\begin{array}{l}
	\forall \vl'\in \Sigma_{\sigma^{\mathcal{N}}\omega}\text{ with } |\vl\wedge\vl'|\geq n\\
	{\left|\frac{S_{n}\Phi_{i}(\sigma^{\mathcal{N}}\omega,\vl')}{S_{n}\Psi_{i}(\sigma^{\mathcal{N}}\omega,\vl')}-\beta\right|\leq \varepsilon}
	\end{array}
	\right\}
	\end{equation*}
	and 
	$$E_{i,\beta}(\sigma^{\mathcal{N}}\omega,N,{\varepsilon})=\bigcap_{n\ge N}F_{i,\beta,n}(\sigma^{\mathcal{N}}\omega,{\varepsilon})\ {\rm and }\ E_{i,\beta}(\sigma^{\mathcal{N}}\omega,{\varepsilon})=\bigcup_{N\geq 1}E_{i,\beta}(\sigma^{\mathcal{N}}\omega,N,{\varepsilon}).$$
	
	Using the same method as in \cite[lemma 3.15]{Yuan2017DCDS}, more easily we can get:
	\begin{proposition}\label{initial} {\rm For all  $i\in\N$, for any $\varepsilon>0$, for all $\omega\in \Omega(i)$, the set $E_{i,\alpha_i}(\sigma^{\mathcal{N}}\omega,{\varepsilon})$ has full $\widetilde\mu_{\sigma^{\mathcal{N}}\omega}^{i}$-measure, where $\alpha_i=T_i'(q_i)$ and $P(q\Phi_i-T_i(q)\Psi_i)=0$. Then we can choose $\mathscr{N}_i(\omega)$ large enough such that $\widetilde\mu_{\sigma^{\mathcal{N}}\omega}^{i}(E_{i,\alpha_i}(\sigma^{\mathcal{N}}\omega,\mathscr{N}_i(\omega),{\varepsilon}))>1/2$. Furthermore we can also choose a set $\Omega_i\subset \Omega(i)$ with $\mathbb{P}(\Omega_i)>1/2$ and $\mathcal{N}_i\in\mathbb{N}$ with $\mathscr{N}_i(\omega)\leq \mathcal{N}_i$ for each $\omega\in \Omega_i$, so that $\widetilde\mu_{\sigma^{\mathcal{N}}\omega}^{i}(E_{i,\alpha_i}(\sigma^{\mathcal{N}}\omega,\mathcal{N}_i,{\varepsilon}))>1/2$.
		}
	\end{proposition}
	\begin{remark}
		For any $\vl\in E_{i,\beta}(\sigma^{\mathcal{N}}\omega,\mathcal{N}_i,{\varepsilon})$, for any $n\geq\mathcal{N}_i$, $\forall \vl'\in \Sigma_{\sigma^{\mathcal{N}}\omega}$ with  $|\vl\wedge\vl'|\geq n$
		we have $\left|\frac{S_{n}\Phi_{i}(\sigma^{\mathcal{N}}\omega,\vl')}{S_{n}\Psi_{i}(\sigma^{\mathcal{N}}\omega,\vl')}-\beta\right|\leq \varepsilon$.
		The notation $\vl\wedge\vl'$ means the longest common prefix of $\vl$ and $\vl'$, that is: for $\vl=v_0v_1\cdots v_n\cdots$ and $\vl'=v_0'v_1'\cdots v_n'\cdots$ with $v_n=v_n'$ for $0\leq n\leq p-1$ and $v_{p}\neq v_{p}'$, define $\vl\wedge\vl'=v_0v_1\cdots v_{p-1}$
	\end{remark}
	\begin{notation}
		For any $\omega\in \Omega$, for any $i\in\mathbb{N}$ define 
		$$\theta(i,\omega,1)=\inf\{n\in\mathbb{N}:\sigma^n\in\Omega_i\}$$
		and for any $p\in\mathbb{N}$ with $p>1$ define 
		$$\theta(i,\omega,p)=\inf\{n\in\mathbb{N}: n>\theta(i,\omega,p-1)\text{ and }\sigma^n\in\Omega_i\}$$
	\end{notation}
	
	From Birkhoff ergodic theorem, for any $i\in\mathbb{N}$, for $\mathbb P$-a.e. $\omega\in\Omega$ 
	\begin{equation}\label{p/thta}
	\lim_{p\rightarrow\infty}\frac{p}{\theta(i,\omega,p)}=\mathbb{P}(\Omega_i)>1/2,
	\end{equation}
	then 
	\begin{equation}\label{D-theta/theta}
	\lim_{p\rightarrow\infty}\frac{\theta(i,\omega,p)-\theta(i,\omega,p-1)}{\theta(i,\omega,p-1)}=0.
	\end{equation}
	Since $\mathbb{N}$ is countable, $\mathbb P$-a.e. $\omega\in\Omega$, for any $i\in\mathbb{N}$, equations \eqref{p/thta}and 
	\eqref{D-theta/theta} also hold. There is no trouble that we can assume that they hold for all $\omega\in\Omega$. 
	
	\begin{remark}
		It is important that we introduce the sequences $\{\Psi_i\}_{i\in\mathbb{N}},\{\Phi_i\}_{i\in\mathbb{N}}$ of random H\"{o}lder potentials  to approximate $\Psi$ and $\Phi$. First, this can be done, see the details in \cite{Yuan2017DCDS}. Second, if we define a function $T$ that is the root of $P(q\Phi-T(q)\Psi)$, we just know the function $T$ is concave. But for the function $T_{i}$ which is the root of $P(q\Phi_i-T_i(q)\Psi_i)$, it is not only differentiable but also analytic (see \cite[chapter 9]{MSU} and \cite{Gundlach}). This will provide us lots of good informations. Furthermore, let $q_i$ be the root of $P(q(\Psi_i+\Phi_i))=0$, we can easily get that $\lim_{i\rightarrow\infty}q_i=q_0$. Recall that $q_0$ is the root of $P(q(\Psi+\Phi))=0$.  
	\end{remark}
	
	\section{Upper bound}\label{section:Upper bound}
	It is easy to check:
	\begin{equation*}\label{key}
	W(\phi,\omega)=\cap_{N\in\mathbb{N}}\cup_{n\geq N}\cup_{v\in\Sigma_{\omega,n}}\{x\in U_{\omega}^{v}:\|T_{\omega}^{v}x-z_{\sigma^{|v|}\omega}\|\leq \exp(S_{|v|}\phi(\omega,x))\}.
	\end{equation*}	
	Define $V_{\omega}^v=:\{x\in U_{\omega}^{v}:\|T_{\omega}^{v}x-z_{\sigma^{|v|}\omega}\|\leq\exp(S_{|v|}\phi(\omega,x))\}$,
	using \eqref{distortion} and the same method in proposition~\ref{Length and measure}, we can easily get there exists a sequence $\{\epsilon_n(\omega)\}$ of positive number such that $\epsilon_n(\omega)$ decreasing to $0$ as $n$ increasing to $\infty$ and for any $y\in U_{\omega}^{v}$,
	\begin{equation*}\label{Control length V-1}
	\left|V_{\omega}^v\right|\leq\exp(S_{|v|}(\psi+\phi)(\omega,y)+|v|\epsilon_{|v|}(\omega)) 
	\end{equation*}
	so that for any $\vl\in [v]_{\omega}$
	\begin{equation}\label{Control length V-2}
	\left|V_{\omega}^v\right|\leq\exp(S_{|v|}(\Psi+\Phi)(\omega,\vl)+|v|\epsilon_{|v|}(\omega)) 
	\end{equation}
	\begin{theorem}
		$\dim_H W(\phi,\omega)\leq q_0$.
	\end{theorem}
	\begin{proof}
		For any $q>0$ and $\delta>0$ denote by $\mathcal{H}_{\delta}^q$ the $q$-dimensional Hausdorff pre-measure computed using coverings by set of diameter less than $\delta$. For any $N\in\mathbb{N}$, define $\delta_N:=\sup_{n\geq N}\sup_{v\in \Sigma_{\omega,n}}{|V_{\omega}^v|}$. Then for any $M\geq N$
		$$\mathcal{H}_{\delta_n}^q(W(\phi,\omega))\leq \sum_{n\geq M}\sum_{v\in \Sigma_{\omega,n}}{|V_{\omega}^v|^q}.$$
		Since \eqref{Control length V-2}, we have
		$$\sum_{v\in \Sigma_{\omega,n}}{|V_{\omega}^v|^q}\leq \sum_{v\in\Sigma_{\omega,n}}\exp(qS_{n}(\Psi+\Phi)(\omega,\vl)+nq\epsilon_{n}(\omega).$$
		If $q>q_0$, $P(q(\Psi+\Phi))<0$.  From lemma \ref{converge Pressure},  for $\mathbb{P}$-almost every  $\omega\in \Omega$, there exists $N'(\omega)\in\mathbb{N}$, for $n>N'(\omega)$ we get  $$\sum_{v\in\Sigma_{\omega,n}}\exp(qS_{n}(\Psi+\Phi)(\omega,\vl)+nq\epsilon_{n}(\omega))\leq \exp(n\frac{P(q(\Psi+\Phi))}{2}).$$
		This implies 
		$$\mathcal{H}_{\delta_n}^q(W(\phi,\omega))\leq \sum_{n\geq \max\{N'(\omega),N\}} \exp(n\frac{P(q(\Psi+\Phi))}{2})<\sum_{n\geq N'(\omega)} \exp(n\frac{P(q(\Psi+\Phi))}{2}),$$
		then $\mathcal{H}^q(W(\phi,\omega))\leq \sum_{n\geq N'(\omega)} \exp(n\frac{P(q(\Psi+\Phi))}{2})<\infty$.
		So $\dim_H W(\phi,\omega)\leq q$. Since $q>q_0$ is arbitrary, we get that $\dim_H X_{\omega}\leq q_0$.
	\end{proof}

	\section[Lower bound]{Lower bound}\label{section:Lower bound}
	
	\begin{theorem}\label{theorem lower bound}
		For $\mathbb{P}$-almost every $\omega\in\Omega$, there exists a set $K_{\omega}\subset W(\phi,\omega)$ and a probability measure $\eta_{\omega}$ on it such that $\dim_{H}\eta_{\omega}\geq q_0$. Then we have $\dim_H W(\phi,\omega)>q_0$.  
	\end{theorem}
	\begin{proof}
		From \eqref{Control upsilon upper} and \eqref{Control Psi lower}, we know that for $\mathbb{P}$-a.e. $\omega\in\Omega$, there exists $\mathcal{K}=\mathcal{K}(\omega)\in\mathbb{N}$ such that for all $n\geq \mathcal{K}$, for all $v\in \Sigma_{\Omega,n}$ we have 
		\begin{equation}\label{Control U--upper--lower}
		\exp(-2Cn)\leq |U_{\omega}^v|\leq \exp(-cn/2)
		\end{equation}
		
		We now fixed a sequence $\{\varepsilon_{i}\}_{i\in\mathbb{N}}$ of positive numbers deceasing to $0$ such that for any $i\in\mathbb{N}$, one has:
		\begin{equation}\label{Control varepsilon 1}
		2(18+4(\alpha_{i}+\varepsilon_{i}))\varepsilon_{i})/c<1,
		\end{equation}
		and
		\begin{equation}\label{Control varepsilon 2}
		\varepsilon_{i}<\min\{\frac{1}{d+2},\frac{1}{q_{i}(2+\alpha_{i})+d+1}\}.
		\end{equation}	
		\begin{description}
			\item[Step1]  Define $Y(\omega)=S_{\theta(1,\omega,1)}\|\Psi(\omega)\|_{\infty}+S_{\theta(1,\omega,1)}\|\Phi(\omega)\|_{\infty}.$
			choose $m=\theta(1,\omega,s)\in \mathbb{N}$ with $s\in\mathbb{N}$ large enough such that 
			\begin{itemize}
				\item $m\geq \mathcal{K}(\omega)$,
				\item $\mathcal N+\mathcal{N}_{1}\leq m\varepsilon^3_{i+1}$,
				\item for all $k\geq m$,
				\begin{equation*}\label{Control epsilon-Psi-1}
				\max\{\epsilon(\Psi,\omega,k),\epsilon(\Phi,\omega,k)\}\leq \varepsilon^3_{1},
				\end{equation*}
			\end{itemize}
			Fix a word $w\in\Sigma_{\omega,m}$, then $\sigma^{m}\omega\in\Omega_{1}$. 
			
			Now, define
			\begin{equation*}\label{E_{w}-1}
			E(\omega,w)=\left\{\pi_\omega(w\ast \vl):\vl\in E_{1,\alpha_{1}}(\sigma^{m+\mathcal{N}}\omega,\mathcal{N}_{1},{\varepsilon_{1}})
			\right\}
			\end{equation*}
			and noticing $\widetilde\mu^{1}_{\sigma^{m+\mathcal N}\omega}(E_{1,\alpha_{1}}(\sigma^{m+\mathcal{N}}\omega,\mathcal{N}_{1},{\varepsilon_{1}})>1/2$ by proposition~\ref{initial}, define
			\begin{equation*}\label{def tzeta-1}
			\widetilde\zeta_{\omega,w}([w\ast v]_{\omega})=\frac{\widetilde\mu^{1}_{\sigma^{m+\mathcal N}\omega}([v]_{\sigma^{m+\mathcal N}\omega}\cap E_{1,\alpha_{1}}(\sigma^{m+\mathcal{N}}\omega,\mathcal{N}_{1},{\varepsilon_{1}})}{\widetilde\mu^{1}_{\sigma^{m+\mathcal N}\omega}(E_{1,\alpha_{1}}(\sigma^{m+\mathcal{N}}\omega,\mathcal{N}_{1},{\varepsilon_{1}})},
			\end{equation*}
			and 
			\begin{equation*}\label{def zeta-1}
			\zeta_{\omega,w}={\pi_\omega}_*\widetilde\zeta_{\omega,w}=\widetilde\zeta_{\omega,w}\circ{\pi_\omega}^{-1}.
			\end{equation*} 	 	
			We will easily obtain that $\zeta_{\omega,w}(E(\omega,w))=1$.
			
			For $n\geq \mathcal{N}_{1}$, define
			$$\mathcal{F}_{w,n}=\{B(x,2|U_{\omega}^{w\ast v}|): x\in E(\omega,w)\cap U_{\omega}^{w\ast v}, v\in\Sigma_{\sigma^{m+\mathcal N},n}\}.$$
			From Besicovitch covering theorem $Q(d)$ families of disjoint balls namely
			$$\mathcal{F}^1_{w,n},\cdots,\mathcal{F}^{Q(d)}_{w,n}$$ can be extracted from $\mathcal{F}_{w,n}$, so that 
			$$
			E(w)\subset \bigcup_{j=1}^{Q(d)} \bigcup_{B\in \mathcal{F}^j_{w,n}}B.
			$$ Since $\zeta_{\omega,w}(E(w))=1$, there exists $j$ such that
			\begin{equation*}\label{}
			\zeta_{\omega,w}\Big(\bigcup_{B\in{\mathcal{F}^j_{w,n}}}B\Big)\geq \frac{1}{Q(d)}.
			\end{equation*}
			
			Again, we extract from $\mathcal{F}^j_{w,n}$ a finite family of pairwise disjoint balls $D(w,n)=\{B_1,\cdots,B_{j'}\}$ such that
			\begin{equation*}\label{Control zeta all ball-1}
			\zeta_{\omega,w}\Big(\bigcup_{B_l\in D(w,n)}B_l\Big)\geq \frac{1}{2Q(d)}.
			\end{equation*}
			For each $B_l\in D(w,n)$, there exists $y_l\in U_{\omega}^{w\ast v(l)}\cap E(w)$ such that $$B_l=B\big(y_l,2|U_{\omega}^{w\ast v(l)}|\big)\supset U_{\omega}^{w\ast v(l)}.$$
			Now we can get (see lemma~\ref{lemma: Control zeta Bxr} in general situation):
			
			For any $x\in E(\omega,w)$, for $r\leq |U_{\omega}^w| \exp(-(C+4)m\varepsilon_{1}^3)$,
			one has that 
			\begin{equation}\label{Control zeta Bxr-1}
			\zeta_{\omega,w}(B(x,r))\leq (3r^{-\varepsilon_{1}^3})^d (\frac{r}{|U_{\omega}^{w}|})^{q_{1}(1+\alpha_{1}-2\varepsilon_{1})}.
			\end{equation}

			Now choose $p_{1}\geq \mathcal N_{1}$ large such that 
			\begin{itemize}
				\item $\sigma^{m+\mathcal{N}+p_{1}}\omega\in\Omega_{1}$,
				\item for all $k\geq m+\mathcal{N}+p_1$,
				\begin{equation*}
				\max\{\epsilon(\Psi,\omega,k),\epsilon(\Phi,\omega,k)\}\leq \varepsilon^3_{2},
				\end{equation*}
				\item 
				\begin{equation*}\label{p1 large 1}
				p_{1}\geq \frac{(m+2\mathcal{N})C+Y(\omega)}{c\varepsilon_{1}^3/2},
				\end{equation*}
				and
				\begin{equation*}\label{p1 large 2}
				p_{1}\geq \frac{(q_{1}(1+\alpha_{1}-2\varepsilon_{1})+d+1)\log 2+\log(Q(d))}{c\varepsilon_{1}^3/2}),
				\end{equation*}
				\item $\mathcal N+\mathcal{N}_{2}\leq (m+\mathcal N+p_{1})\varepsilon^3_{2}$, that is $m+\mathcal N+p_{1}\geq    \frac{\mathcal N+\mathcal{N}_{2}}{\varepsilon^3_{2}}$, 
				\item for any $k\in\mathbb{N}$ such that $\theta(2,\omega,k)\geq m+\mathcal N+p_{1}$, we have 
				\begin{equation*}\label{control theta Diff-pre-1}
				\frac{\theta(2,\omega,k)-\theta(2,\omega,k-1)}{\theta(2,\omega,k-1)}\leq \varepsilon^3_{2}
				\end{equation*}
			\end{itemize}
			%

				Noticing that there exists $1\leq s\leq l(\sigma^{m+\mathcal{N}+p_{1}+\mathcal{N}}\omega)$ such that $$z_{\sigma^{m+\mathcal{N}+p_{1}+\mathcal{N}}\omega}\in X_{\sigma^{m+\mathcal{N}+p_{1}+\mathcal{N}}\omega}^{s}\subset X_{\sigma^{m+\mathcal{N}+p_{1}+\mathcal{N}}\omega}.$$
			Since $\sigma^{m+\mathcal N+p_{1}-1}\omega\in \Omega_{1}$, we know that there exists $w\ast v(l)\ast s\in\Sigma_{\omega,m+\mathcal{N}+p_{1}+\mathcal{N}+1}$ such that $y=:g_{\omega}^{w\ast v(l)\ast}(z_{\sigma^{m+\mathcal{N}+p_{1}+\mathcal{N}}\omega})\in X_{\omega}^{w\ast v(l)\ast s}\subset U_{\omega}^{w\ast v(l)\ast s}$. 
			
			Now we can choose the smallest $p$ and $v'(l)\in \Sigma_{\sigma^{m+\mathcal{N}}\omega,\theta(2,\omega,p)-m-\mathcal{N}}$ such that
			\begin{itemize}
				\item $v(l)\ast s$ is the prefix of $v'(l)$, which is equivalent to $U_{\omega}^{w\ast v'(l)}\subset U_{\omega}^{w\ast v(l)\ast s},$
				\item $y\in U_{\omega}^{w\ast v'(l)}$
				\item  define $r_\omega^{w\ast v\ast}=\exp(S_{m+\mathcal{N}+p_{1}+\mathcal{N}}\phi(\omega,y)-(m+\mathcal{N}+p_{1}+\mathcal{N})\varepsilon_{1}^3)$, then 
				\begin{equation*}\label{Choice of v'-1}
				T_{\omega}^{w\ast v(l)\ast}(U_{\omega}^{w\ast v'(l)})\subset B(z_{\sigma^{m+\mathcal{N}+p_{1}+\mathcal{N}}\omega},r_\omega^{w\ast v\ast}).
				\end{equation*}
			\end{itemize}
			Then for any $x\in U_{\omega}^{w\ast v'(l)}\subset U_{\omega}^{w\ast v(l)\ast s}$, 
			\begin{equation*}\label{Condition Tnx to z-1}
			|T_{\omega}^{w\ast v(l)\ast}x-z_{\sigma^{m+\mathcal{N}+p_{1}+\mathcal{N}}\omega}|\leq \exp(S_{m+\mathcal{N}+p_{1}+\mathcal{N}}\phi(\omega,x)),
			\end{equation*}	
			and
			\begin{equation}\label{Control Uwv'-1}
			|U_{\omega}^{w\ast v'(l)}|\geq |U_{\omega}^{w\ast v(l)}|^{1+\alpha_{1}+2\varepsilon_{1}}, 
			\end{equation} 
			(see \eqref{Condition Tnx to z} and \eqref{Control Uwv'} in general situation) 
			
			Let $G(\omega,w)=\{U_{\omega}^{w\ast v'(l)}:B_l\in D^{w,p_{1}}\}$, $G_{\omega,1}=G(\omega,w)$ and define a set function $\eta_{\omega}^{1}$ as follows,
			\begin{equation*}\label{define m(1)}
			\eta_{\omega}^{1}(U_{\omega}^{w\ast v'(l)})=\frac{\zeta_{\omega,w}(B_l)}{\sum_{B\in G(w)}\zeta_{\omega,w}(B)}.
			\end{equation*}
			From \eqref{Control zeta Bxr-1} and \eqref{Control Uwv'-1}, we will get (see \eqref{Control eta Uwv'} in general situation): 
			$$\eta_{\omega}^{1}(U_{\omega}^{w\ast v'(l)})\leq |U_{\omega}^{w\ast v'(l)}|^{q_{1}(1-4\varepsilon_{1})-\varepsilon_{1}^2}$$

			\item[Step 2] 
			Suppose that $G_{\omega,i}$ is well defined and so is the set function $\eta_{\omega}^{i}$ on it.
			For any $w$ such that $U^w_\omega\in G_{\omega,i}$, set $m=|w|$, then $\sigma^{m}\omega\in\Omega_{i+1}$ and the following hold
			\begin{itemize}
				\item $\mathcal N+\mathcal{N}_{i+1}\leq m\varepsilon^3_{i+1}$,
				\item for all $k\geq m$,
				\begin{equation}\label{Control epsilon-Psi}
				\max\{\epsilon(\Psi,\omega,k),\epsilon(\Phi,\omega,k)\}\leq \varepsilon^3_{i+1},
				\end{equation}
			\end{itemize} 
			Now, we define    
			\begin{equation*}\label{E_{w}}
			E(\omega,w)=\left\{\pi_\omega(w\ast \vl)\in U_{\omega}^{w}:\vl\in E_{i+1,\alpha_{i+1}}(\sigma^{m+\mathcal{N}}\omega,\mathcal{N}_{i+1},{\varepsilon_{i+1}})
			\right\}
			\end{equation*}
			and noticing $\widetilde\mu^{i+1}_{\sigma^{m+\mathcal N}\omega}(E_{i+1,\alpha_{i+1}}(\sigma^{m+\mathcal{N}}\omega,\mathcal{N}_{i+1},{\varepsilon_{i+1}})>1/2$ in proposition~\ref{initial}, we can define
			\begin{equation*}\label{def tzeta}
			\widetilde\zeta_{\omega,w}([w\ast v]_{\omega})=\frac{\widetilde\mu^{i+1}_{\sigma^{m+\mathcal N}\omega}([v]_{\sigma^{m+\mathcal N}\omega}\cap E_{i+1,\alpha_{i+1}}(\sigma^{m+\mathcal{N}}\omega,\mathcal{N}_{i+1},{\varepsilon_{i+1}})}{\widetilde\mu^{i+1}_{\sigma^{m+\mathcal N}\omega}(E_{i,\alpha_{i+1}}(\sigma^{m+\mathcal{N}}\omega,\mathcal{N}_{i+1},{\varepsilon_{i+1}})},
			\end{equation*}
			and 
			\begin{equation*}\label{def zeta}
			\zeta_{\omega,w}={\pi_\omega}_*\widetilde\zeta_{\omega,w}=\widetilde\zeta_{\omega,w}\circ{\pi_\omega}^{-1}.
			\end{equation*} 	 	
			We will easily obtain that $\zeta_{\omega,w}(E(\omega,w))=1$.
			
			For any $n\geq \mathcal{N}_{i+1}$, define
			$$\mathcal{F}_{w,n}=\{B(x,2|U_{\omega}^{w\ast v}|): x\in E(\omega,w)\cap U_{\omega}^{w\ast v}, v\in\Sigma_{\sigma^{m+\mathcal N},n}\}.$$
			From Besicovitch covering theorem $Q(d)$ families of disjoint balls namely
			$$\mathcal{F}^1_{w,n},\cdots,\mathcal{F}^{Q(d)}_{w,n}$$ can be extracted from $\mathcal{F}_{w,n}$, so that 
			$$
			E(\omega,w)\subset \bigcup_{j=1}^{Q(d)} \bigcup_{B\in \mathcal{F}^j_{w,n}}B.
			$$ Since $\zeta_{\omega,w}(E(\omega,w))=1$, there exists $j$ such that
			\begin{equation*}\label{}
			\zeta_{\omega,w}\Big(\bigcup_{B\in{\mathcal{F}^j_{w,n}}}B\Big)\geq \frac{1}{Q(d)}.
			\end{equation*}
			
			Again, we extract from $\mathcal{F}^j_{w,n}$ a finite family of pairwise disjoint balls $D(w,n)=\{B_1,\cdots,B_{j'}\}$ such that
			\begin{equation}\label{Control zeta all ball}
			\zeta_{\omega,w}\Big(\bigcup_{B_l\in D(w,n)}B_l\Big)\geq \frac{1}{2Q(d)}.
			\end{equation}
			For each $B_l\in D(w,n)$, there exists $y_l\in U_{\omega}^{w\ast v(l)}\cap E(w)$ such that $$B_l=B\big(y_l,2|U_{\omega}^{w\ast v(l)}|\big)\supset U_{\omega}^{w\ast v(l)}.$$

			Now we turn to prove the following lemma.
			\begin{lemma}\label{lemma: Control zeta Bxr}
				For any $x\in E(w)$, for $r\leq |U_{\omega}^w| \exp(-(C+4)m\varepsilon_{i+1}^3)$,
				one has that 
				$$\zeta_{\omega,w}(B(x,r))\leq (2r^{-\varepsilon_{i+1}^3})^d (\frac{r}{|U_{\omega}^{w}|})^{q_{i+1}(1+\alpha_{i+1}-2\varepsilon_{i+1})}.$$
			\end{lemma}
			\begin{proof}
				Define
				\begin{equation*}\label{Delta_{w,r}}
				\Delta_{w,r}=\left\{U_{\omega}^{w\ast v}:\begin{array}{l}
				[v]_{\sigma^{m+\mathcal{N}}\omega}\cap E_{i,\alpha_{i+1}}(\sigma^{m+\mathcal{N}}\omega,\mathcal{N}_i,{\varepsilon_i})\neq \emptyset\\
				|U_{\omega}^{w\ast v}|< r, |U_{\omega}^{w\ast v^\star}|\geq r
				\end{array}
				\right\},
				\end{equation*}
				recall the definition of $\star$ in subsection~\ref{subsection: Random subshift}.
				Let $k_1=\inf\{|v|,U_{\omega}^{w\ast v}\in \Delta_{w,r}\}$ and $k_2=\sup\{|v|,U_{\omega}^{w\ast v}\in \Delta_{w,r}\}$.
				Noticing the following fact:
				\begin{itemize}
					\item $k_1\ge \mathcal{N}_{i+1}$,
					in fact since $|w\ast v|=m+\mathcal{N}+n$ and $|U_{\omega}^{w\ast v}|< r\leq |U_{\omega}^w| \exp(-(C+4)m\varepsilon_{i+1}^3)$, we will get 
					$$\exp(S_{\mathcal{N}+n}(F^{m}(\omega,\vl))-2(m+\mathcal{N}+n)\epsilon(\Psi,\omega,m+\mathcal{N}+n))< \exp(-(C+4)m\varepsilon_{i+1}^3).$$
					From \eqref{Control upsilon upper} and \eqref{Control epsilon-Psi}, we have 
					$$C(n+\mathcal{N})+2(m+n+\mathcal{N})\varepsilon_{i+1}^3\geq (C+4)m\varepsilon_{i+1}^3,$$
					which yield 
					$n+\mathcal{N}\geq m\varepsilon_{i+1}^3,$
					and then $n\geq \mathcal{N}_{i+1}$ by $\mathcal N+\mathcal{N}_{i+1}\leq m\varepsilon^3_{i+1}$.
					\item from \eqref{var Psiphi-i}, \eqref{Control averagepsi upsilon} and \eqref{Control averagephi upsilon}, we have that for any $\vl\in [w\ast v]_{\omega}$ with $U_{\omega}^{w\ast v}\in \Delta_{w,r}$,
					\begin{equation}\label{Control Diff Phi 0-i}
					|S_{|v|}\Phi(\omega,\vl)-S_{|v|}\Phi_{i+1}(\omega,\vl)|\leq 2|v|\varepsilon_{i+1}^3,
					\end{equation}
					
					\begin{equation}\label{Control Diff Psi 0-i}
					|S_{|v|}\Psi(\omega,\vl)-S_{|v|}\Psi_{i+1}(\omega,\vl)|\leq 2|v|\varepsilon_{i+1}^3
					\end{equation}
					and
					\begin{equation}\label{Control average i}
					\left|\frac{S_{|v|}\Phi_i(\omega,\vl)}{S_{|v|}\Psi_i(\omega,\vl)}-\alpha_{i+1}\right|\leq \varepsilon_{i+1}
					\end{equation}
					
					\item for any $U_{\omega}^{w\ast v}\in \Delta_{w,r}$, using \eqref{Control Diff Phi 0-i},\eqref{Control Diff Psi 0-i} and \eqref{Control average i} and the definition of $\widetilde\zeta_{\omega,w}$ and $\zeta_{\omega,w}$, we can get 
					\begin{equation}\label{control mass u-wv}
					\left|\frac{\log \widetilde\zeta_{\omega,w}([w\ast v]_{\omega})}{\log |U_{\omega}^{w\ast v}|-\log |U_{\omega}^{w}|}-q_{i+1}(1+\alpha_{i+1})\right|\leq 2q_{i+1}\varepsilon_{i+1} 
					\end{equation}
					\item for any $U_{\omega}^{w\ast v}\in \Delta_{w,r}$ with $U_{\omega}^{w\ast v}\cap B(x,r)\neq \emptyset$, we will get $U_{\omega}^{w\ast v}\subset B(x,2r)$. Noticing $|U_{\omega}^{w\ast v^\star}|\geq r$ then $U_{\omega}^{w\ast v}$ contains a ball $B_{w\ast v}$ with radius $r^{1+\varepsilon_{i+1}^3}$. Since $\{B_{w\ast v}\}_{U_{\omega}^{w\ast v}\in \Delta_{w,r}}$ have no inner intersection. So there are at most $(2r^{-\varepsilon_{i+1}^3})^d$ of $U_{\omega}^{w\ast v}\in \Delta_{w,r}$ can intersect with $B(x,r)$.
				\end{itemize} 
				So
				\begin{eqnarray*}
					&&\zeta_{\omega,w}(B(x,r)\leq \widetilde\zeta_{\omega,w}(\pi_\omega^{-1}(B(x,r)))\\
					&\leq&\sum_{U_{\omega}^{w\ast v}\in \Delta_{w,r},U_{\omega}^{w\ast v}\cap B(x,r)\neq \emptyset}\widetilde\zeta_{\omega,w}([w\ast v])\\
					&\leq&(2r^{-\varepsilon_{i+1}^3})^d\sup_{U_{\omega}^{w\ast v}\in \Delta_{w,r}}\widetilde\zeta_{\omega,w}([w\ast v])\\
					&\leq&(2r^{-\varepsilon_{i+1}^3})^d\sup_{U_{\omega}^{w\ast v}\in \Delta_{w,r}}\left(\frac{U_{\omega}^{w\ast v}}{U_{\omega}^{w}}\right)^{q_{i+1}(1+\alpha_{i+1}-2\varepsilon_{i+1})} \text{(see \eqref{control mass u-wv})}\\
					&\leq&(2r^{-\varepsilon_{i+1}^3})^d (\frac{r}{|U_{\omega}^{w}|})^{q_{i+1}(1+\alpha_{i+1}-2\varepsilon_{i+1})}
				\end{eqnarray*}
			\end{proof}

			Now choose $p_{i+1}\geq \mathcal N_{i+1}$ large such that 
			\begin{itemize}
				\item $\sigma^{m+\mathcal{N}+p_{i+1}}\omega\in\Omega_{i+1}$
				\item for all $k\geq m+\mathcal{N}+p_{i+1}$,
				\begin{equation*}
				\max\{\epsilon(\Psi,\omega,k),\epsilon(\Phi,\omega,k)\}\leq \varepsilon^3_{i+2},
				\end{equation*}
				\item 
				\begin{equation}\label{pi+1 large 1}
				p_{i+1}\geq \frac{(m+2\mathcal{N})C+Y(\omega)}{c\varepsilon_{i+1}^3/2},
				\end{equation}
				and
				\begin{equation}\label{pi+1 large 2}
				p_{i+1}\geq \frac{(q_{i+1}(1+\alpha_{i+1}-2\varepsilon_{i+1})+d+1)\log 2+\log (Q(d))}{c\varepsilon_{i+1}^3/2},
				\end{equation}
				\item $\mathcal N+\mathcal{N}_{i+2}\leq (m+\mathcal N+p_{i+1})\varepsilon^3_{i+2}$, that is $m+\mathcal N+p_{i+1}\geq    \frac{\mathcal N+\mathcal{N}_{i+2}}{\varepsilon^3_{i+2}}$, 
				\item for any $k\in\mathbb{N}$ such that $\theta(i+2,\omega,k)\geq m+\mathcal N+p_{i+1}$, we have 
				\begin{equation}\label{control theta Diff-pre}
				\frac{\theta(i+2,\omega,k)-\theta(i+2,\omega,k-1)}{\theta(i+2,\omega,k-1)}\leq \varepsilon^3_{i+2}
				\end{equation}
			\end{itemize}
			Noticing that there exists $1\leq s\leq l(\sigma^{m+\mathcal{N}+p_{i+1}+\mathcal{N}}\omega)$ such that $$z_{\sigma^{m+\mathcal{N}+p_{i+1}+\mathcal{N}}\omega}\in X_{\sigma^{m+\mathcal{N}+p_{i+1}+\mathcal{N}}\omega}^{s}\subset X_{\sigma^{m+\mathcal{N}+p_{i+1}+\mathcal{N}}\omega}.$$
			Since $\sigma^{m+\mathcal N+p_{i+1}-1}\omega\in \Omega_{i+1}$, we know that there exists $$w\ast v(l)\ast s\in\Sigma_{\omega,m+\mathcal{N}+p_{i+1}+\mathcal{N}+1}$$ such that $$y=:g_{\omega}^{w\ast v(l)\ast}(z_{\sigma^{m+\mathcal{N}+p_{i+1}+\mathcal{N}}\omega})\in X_{\omega}^{w\ast v(l)\ast s}\subset U_{\omega}^{w\ast v(l)\ast s}.$$
			
			Now we can choose the smallest $p$ and $v'(l)\in \Sigma_{\sigma^{m+\mathcal{N}}\omega,\theta(i+2,\omega,p)-m-\mathcal{N}}$ such that
			\begin{itemize}
				\item $v(l)\ast s$ is the prefix of $v'(l)$, which is equivalent to $U_{\omega}^{w\ast v'(l)}\subset U_{\omega}^{w\ast v(l)\ast s},$
				\item $y\in U_{\omega}^{w\ast v'(l)}$
				\item  define $r_\omega^{w\ast v\ast}=\exp(S_{m+\mathcal{N}+p_{i+1}+\mathcal{N}}\phi(\omega,y)-(m+\mathcal{N}+p_{i+1}+\mathcal{N})\varepsilon_{i+1}^3)$, then 
				\begin{equation}\label{Choice of v'}
				T_{\omega}^{w\ast v(l)\ast}(U_{\omega}^{w\ast v'(l)})\subset B(z_{\sigma^{m+\mathcal{N}+p_{i+1}+\mathcal{N}}\omega},r_\omega^{w\ast v\ast}).
				\end{equation}
			\end{itemize}
			We claim that: for any $x\in U_{\omega}^{w\ast v'(l)}\subset U_{\omega}^{w\ast v(l)\ast s}$, 
			\begin{equation}\label{Condition Tnx to z}
			|T_{\omega}^{w\ast v(l)\ast}x-z_{\sigma^{m+\mathcal{N}+p_{i+1}+\mathcal{N}}\omega}|\leq \exp(S_{m+\mathcal{N}+p_{i+1}+\mathcal{N}}\phi(\omega,x)).
			\end{equation}
			
			In fact, \eqref{Choice of v'} implies $|T_{\omega}^{w\ast v(l)\ast}x-z_{\sigma^{m+\mathcal{N}+p_{i+1}+\mathcal{N}}\omega}|\leq r_\omega^{w\ast v\ast}$, then noticing $x,y\in U_{\omega}^{w\ast v(l)\ast s}$, \eqref{Control V-n-psi} and \eqref{Control epsilon-Psi}, we have $$ S_{m+\mathcal{N}+p_{i+1}+\mathcal{N}}\phi(\omega,y)-(m+\mathcal{N}+p_{i+1}+\mathcal{N})\varepsilon_{i+1}^3\leq S_{m+\mathcal{N}+p_{i+1}+\mathcal{N}}\phi(\omega,x),$$
			from the definition of $r_\omega^{w\ast v\ast}$, the claim holds.
			
			\begin{lemma}\label{lemma:Control Uwv'}
				\begin{equation}\label{Control Uwv'}
				|U_{\omega}^{w\ast v'(l)}|\geq |U_{\omega}^{w\ast v(l)}|^{1+\alpha_{i+1}+2\varepsilon_{i+1}} 
				\end{equation} 
			\end{lemma}
		 Although the proof of the lemma is a bit complex, but the main idea is very simple. First, since we have choose the smallest $p$, and $\theta(i+2,\omega,p)-\theta(i+2,\omega,p-1)$ is smaller than $\theta(i+2,\omega,p)\varepsilon_{i+1}^3$. Second, use \eqref{Control Diff Phi 0-i},\eqref{Control Diff Psi 0-i} and \eqref{Control average i}. We can skip the details of the proof at first sight.
			\begin{proof}
				From the choice of $p$ and $v'(l)$, define $v''(l)=v'(l)|_{\theta(i+2,\omega,p-1)-m-\mathcal{N}}$.
				We claim that 
				\begin{equation}\label{Control length v&v'}
				|U_{\omega}^{w\ast v'(l)}|\geq |U_{\omega}^{w\ast v''(l)}|^{1+\varepsilon_{i+1}^2}
				\end{equation}
				In fact, for any $\vl\in [w\ast v'(l)]_{\omega}$ 
				\begin{align*}
				|U_{\omega}^{w\ast v'(l)}|&\geq\exp(S_{\theta(i+2,\omega,p)}(\omega,\vl)-\theta(i+2,\omega,p)\epsilon(\Psi,\omega,\theta(i+2,\omega,p)))\\
				&\geq\exp(S_{\theta(i+2,\omega,p-1)}(\omega,\vl)-(\theta(i+2,\omega,p)-\theta(i+2,\omega,p-1))C)\\
				&\quad\cdot\exp(-\theta(i+2,\omega,p)\varepsilon_{i+1}^3)\\
				&\geq\exp(S_{\theta(i+2,\omega,p-1)}(\omega,\vl)-C\theta(i+2,\omega,p-1)\varepsilon_{i+1}^3)\\
				&\cdot\exp(-\theta(i+2,\omega,p-1)(1+\varepsilon_{i+1}^3)\varepsilon_{i+1}^3)
				\quad(\text{see \eqref{control theta Diff-pre}})\\
				&\geq|U_{\omega}^{w\ast v''(l)}|\exp(-\theta(i+2,\omega,p)\epsilon(\Psi,\omega,\theta(i+2,\omega,p))\\
				&\quad \exp(-(C\theta(i+2,\omega,p-1)+\theta(i+2,\omega,p-1)(1+\varepsilon_{i+1}^3))\varepsilon_{i+1}^3)\\
				&\geq |U_{\omega}^{w\ast v''(l)}|\exp(-(C+3)\theta(i+2,\omega,p-1)\varepsilon_{i+1}^3)\\
				&\text{ since } \epsilon(\Psi,\omega,\theta(i+2,\omega,p)<\varepsilon_{i+1}^3 \text{ and }  \varepsilon_{i+1}^3<1\\
				&\geq  |U_{\omega}^{w\ast v''(l)}|^{1+\frac{(C+3)\varepsilon_{i+1}^3}{c/2}}\geq |U_{\omega}^{w\ast v''(l)}|^{1+\varepsilon_{i+1}^2}\ \text{see \eqref{Control U--upper--lower}.}
				\end{align*}
				
				Now there will be {\bf two cases}:
				
				{\bf case 1.} $U_{\omega}^{w\ast v''(l)}\not\subset U_{\omega}^{w\ast v(l)\ast s}$. Since $w\ast v''(l)$ and $w\ast v$ are the prefixes of $w\ast v'(l)$, we have $U_{\omega}^{w\ast v''(l)}\supset U_{\omega}^{w\ast v(l)\ast s}$, using a similar method as the proof of \eqref{Control length v&v'}, we can get 
				$$|U_{\omega}^{w\ast v(l)\ast s}|\geq |U_{\omega}^{w\ast v(l)}|^{1+\varepsilon_{i+1}^2},$$
				now \eqref{Control length v&v'} can imply 
				$$|U_{\omega}^{w\ast v'(l)}|\geq |U_{\omega}^{w\ast v''(l)}|^{1+\varepsilon_{i+1}^2}\geq |U_{\omega}^{w\ast v(l)}|^{({1+\varepsilon_{i+1}^2})^2}\geq   |U_{\omega}^{w\ast v(l)}|^{1+\alpha_{i+1}+2\varepsilon_{i+1}}.$$

				{\bf Case 2.}  $U_{\omega}^{w\ast v''(l)}\subset U_{\omega}^{w\ast v(l)\ast s}$. Since we have chosen the smallest $p$, so from \eqref{Choice of v'}, there exist $x\in U_{\omega}^{w\ast v''(l)}$ such that $$|T_{\omega}^{w\ast v(l)\ast}x-z_{\sigma^{m+\mathcal{N}+p_{i+1}+\mathcal{N}}\omega}|> r_{\omega}^{w\ast v(l)\ast}.$$
				From \eqref{distortion}, noticing $y=g_{\omega}^{w\ast v(l)\ast}(z_{\sigma^{m+\mathcal{N}+p_{i+1}+\mathcal{N}}\omega})$, $x=g_{\omega}^{w\ast v(l)\ast}(T_{\omega}^{w\ast v(l)\ast}x)$,
				for any $\vl\in [w\ast v(l)\ast]_{\omega}$ and \eqref{Control epsilon-Psi}, we will get 
				\begin{equation*}\label{control d(x,y)}
				|U_{\omega}^{w\ast v''(l)}|\geq\|x-y\|\geq \exp(S_{m+p_{i+1}+2\mathcal{N}}(\Psi+\Phi)(\omega,\vl)-4(m+p_{i+1}+2\mathcal{N})\varepsilon_{i+1}^3),
				\end{equation*}
				so that we can  claim that
				\begin{equation*}\label{control U-wv''}
				|U_{\omega}^{w\ast v''(l)}|\geq\exp(S_{p_{i+1}}(\Psi+\Phi)F^{m+\mathcal{N}}(\omega,\vl))-5(m+p_{i+1}+2\mathcal{N})\varepsilon_{i+1}^3).
				\end{equation*}
				
				In fact we just need to notice the following claim: for $\Upsilon\in\{\Psi,\Psi+\Phi\}$ we have 
				\begin{equation}\label{Diff m+p to p}
				|S_{m+p_{i+1}+2\mathcal{N}}\Upsilon(\omega,\vl)-S_{p_{i+1}}\Upsilon(F^{m+\mathcal{N}}(\omega,\vl))|\leq p_{i+1}\varepsilon_{i+1}^3,
				\end{equation}
				It is just from \eqref{Control upsilon upper} and \eqref{pi+1 large 1}, then
				\begin{align*}
				&|S_{m+p_{i+1}+2\mathcal{N}}\Upsilon(\omega,\vl)-S_{p_{i+1}}\Upsilon(F^{m+2\mathcal{N}}(\omega,\vl))|\\
				=&|S_{\theta(1,\omega,1)}\Upsilon(\omega,\vl)+S_{m+\mathcal{N}-\theta(1,\omega,1)}\Upsilon(F^{\theta(1,\omega,1)}(\omega,\vl))+S_{\mathcal{N}}\Upsilon(F^{m+p_{i+1}+\mathcal{N}}(\omega,\vl))|\\
				\leq& S_{\theta(1,\omega,1)}\|\Upsilon(\omega)\|_{\infty}+(m+2\mathcal{N}C)\ \text{see \eqref{Control upsilon upper}}\\
				\leq&p_{i+1}\varepsilon_{i+1}^3 \ \text{see \eqref{pi+1 large 1}}.
				\end{align*}
				%
				Now
				\begin{align*}
				|U_{\omega}^{w\ast v''(l)}|
				&\geq \exp((S_{p_{i+1}}(\Psi+\Phi)(F^{m+\mathcal{N}}(\omega,\vl))-5(m+p_{i+1}+2\mathcal{N})\varepsilon_{i+1}^3)\\
				&\geq\exp(-10p_{i+1}\varepsilon_{i+1}^3+(S_{p_{i+1}}(\Psi+\Phi)(F^{m+\mathcal{N}}(\omega,\vl))))\qquad(\text{since } p_{i+1}\geq m+2\mathcal{N})\\
				&\geq \exp(-10p_{i+1}\varepsilon_{i+1}^3+(S_{p_{i+1}}(\Psi_{i+1}+\Phi_{i+1})(F^{m+\mathcal{N}}(\omega,\vl))-4p_{i+1}\varepsilon_{i+1}^3))\\
				&\text{ see \eqref{Control Diff Phi 0-i} and \eqref{Control Diff Psi 0-i}}\\
				&\geq \exp(-14p_{i+1}\varepsilon_{i+1}^3+(1+\alpha_{i+1}+\varepsilon_{i+1})(S_{p_{i+1}}(\Psi_{i+1})(F^{m+\mathcal{N}}(\omega,\vl))))\\
				&\text{see \eqref{Control average i}}\\
				&\geq \exp(-14p_{i+1}\varepsilon_{i+1}^3+(1+\alpha_{i+1}+\varepsilon_{i+1})(S_{p_{i+1}}\Psi(F^{m+\mathcal{N}}(\omega,\vl))-p_{i+1}\varepsilon_{i+1}^3)\\
				&\text{see \eqref{Control Diff Psi 0-i}}\\
				&\geq \exp((1+\alpha_{i+1}+\varepsilon_{i+1})(S_{m+\mathcal{N}+p_{i+1}}\Psi(\omega,\vl)-2p_{i+1}\varepsilon_{i+1}^3)-14p_{i+1}\varepsilon_{i+1}^3)\\
				&\text{(see \eqref{Diff m+p to p}, but change $2\mathcal{N}$ to $\mathcal{N}$)}\\
				&\geq\exp((1+\alpha_{i+1}+\varepsilon_{i+1})(S_{m+\mathcal{N}+p_{i+1}}(\Psi)(\omega,\vl))-(16+2(\alpha_{i+1}+\varepsilon_{i+1}))p_{i+1}\varepsilon_{i+1}^3).
				\end{align*}
				Noticing item 1 of proposition~\ref{Length and measure} and \eqref{Control epsilon-Psi}, we have 
				$$|U_{\omega}^{w\ast v(l)}|\leq\exp(S_{m+\mathcal{N}+p_{i+1}}(\Psi)(\omega,\vl)+(m+\mathcal{N}+p_{i+1})\varepsilon_{i+1}^3)\leq \exp(-cp_{i+1}/2).$$
				Then
				\begin{align*}
				|U_{\omega}^{w\ast v''(l)}|
				\geq& |U_{\omega}^{w\ast v(l)}|^{(1+\alpha_{i+1}+\varepsilon_{i+1})}\exp(-(1+\alpha_{i+1}+\varepsilon_{i+1})(m+\mathcal{N}+p_{i+1})\varepsilon_{i+1}^3)\\
				&\cdot\exp(-(16+2(\alpha_{i+1}+\varepsilon_{i+1}))p_{i+1}\varepsilon_{i+1}^3)  \\
				\geq&|U_{\omega}^{w\ast v(l)}|^{(1+\alpha_{i+1}+\varepsilon_{i+1})}\exp(-(18+4(\alpha_{i+1}+\varepsilon_{i+1}))p_{i+1}\varepsilon_{i+1}^3)\\
				\geq&|U_{\omega}^{w\ast v(l)}|^{(1+\alpha_{i+1}+\varepsilon_{i+1})+2(18+4(\alpha_{i+1}+\varepsilon_{i+1}))\varepsilon_{i+1}^3)/c}\ \text{see \eqref{Control U--upper--lower}}\\
				\geq& |U_{\omega}^{w\ast v(l)}|^{1+\alpha_{i+1}+\varepsilon_{i+1}+\varepsilon_{i+1}^2}\\
				&\text{see \eqref{Control varepsilon 1}.}\\
				\end{align*}
				That is 
				\begin{equation*}\label{Control Uwv''}
				|U_{\omega}^{w\ast v''(l)}|\geq |U_{\omega}^{w\ast v(l)}|^{1+\alpha_{i+1}+\varepsilon_{i+1}+\varepsilon_{i+1}^2}, 
				\end{equation*}
				also by \eqref{Control length v&v'}, we have 
				$$|U_{\omega}^{w\ast v'(l)}|\geq |U_{\omega}^{w\ast v(l)}|^{(1+\varepsilon_{i+1}^2)(1+\alpha_{i+1}+\varepsilon_{i+1}+\varepsilon_{i+1}^2)}\geq |U_{\omega}^{w\ast v(l)}|^{1+\alpha_{i+1}+2\varepsilon_{i+1}}.$$
			\end{proof}

			Let $G(\omega,w)=\{U_{\omega}^{w\ast v'(l)}:B_l\in D(w,p_{i+1})\}$, $G_{\omega,i+1}=\bigcup_{U_{\omega}^{w}\subset G_{\omega,i}}G(\omega,w)$ and define $\eta_{\omega}^{i+1}$ as follows,
			\begin{equation*}\label{define m(i+1)}
			\eta_{\omega}^{i+1}(U_{\omega}^{w\ast v'(l)})=\frac{\zeta_{\omega,w}(B_l)}{\sum_{B\in G(w)}\zeta_{\omega,w}(B)}\left(\eta_{\omega}^{i}(U_{\omega}^w)\right).
			\end{equation*}
			Now we turn to estimate $\eta_{\omega}^{i+1}(U_{\omega}^{w\ast v'(l)})$.

			Lemma~\ref{lemma: Control zeta Bxr} and \eqref{Control zeta all ball} tells us that 
			\begin{equation*}
			\eta_{\omega}^{i+1}(U_{\omega}^{w\ast v'(l)})\leq 2Q(d)\eta_{\omega}^{i}(U_{\omega}^{w})(4|U_{\omega}^{w\ast v(l)}|^{-\varepsilon_{i+1}^3})^d (\frac{2|U_{\omega}^{w\ast v(l)}|}{|U_{\omega}^w|})^{q_{i+1}(1+\alpha_{i+1}-2\varepsilon_{i+1})},
			\end{equation*}   
			then using lemma~\ref{lemma:Control Uwv'},
			\begin{align*}
			&\eta_{\omega}^{i+1}(U_{\omega}^{w\ast v'(l)})\\
			\leq &2Q(d)\eta_{\omega}^{i}(U_{\omega}^{w})(2(2|U_{\omega}^{w\ast v(l)}|)^{-\varepsilon_{i+1}^3})^d (\frac{2|U_{\omega}^{w\ast v(l)}|}{|U_{\omega}^w|})^{q_{i+1}(1+\alpha_{i+1}-2\varepsilon_{i+1})}\\
			\leq&2^{q_{i+1}(1+\alpha_{i+1}-2\varepsilon_{i+1})+d+1}Q(d) |U_{\omega}^{w\ast v(l)}|^{q_{i+1}(1+\alpha_{i+1}-2\varepsilon_{i+1})-d\varepsilon_{i+1}^3} |U_{\omega}^w|^{-q_{i+1}(1+\alpha_{i+1}-2\varepsilon_{i+1})}.
			\end{align*}
			Noticing \eqref{pi+1 large 1} and \eqref{pi+1 large 2}, we have
			\begin{eqnarray*}
				&\frac{\log (2^{q_{i+1}(1+\alpha_{i+1}-2\varepsilon_{i+1})+d+1}Q(d))}{\log |U_{\omega}^{w\ast v(l)}|}\leq \varepsilon_{i+1}^3,\\
				&\frac{\log |U_{\omega}^w|}{\log |U_{\omega}^{w\ast v(l)}|}\leq \frac{Y(\omega)+C(m+\mathcal{N})}{cp_{i+1}/2}\leq \varepsilon_{i+1}^3,
			\end{eqnarray*}
			then
			\begin{align*}
			&\eta_{\omega}^{i+1}(U_{\omega}^{w\ast v'(l)})\\
			\leq&|U_{\omega}^{w\ast v(l)}|^{q_{i+1}(1+\alpha_{i+1}-2\varepsilon_{i+1})-d\varepsilon_{i+1}^3-q_{i+1}(1+\alpha_{i+1}-2\varepsilon_{i+1})\varepsilon_{i+1}^3-\varepsilon_{i+1}^3}\\
			\leq& |U_{\omega}^{w\ast v(l)}|^{q_{i+1}(1+\alpha_{i+1}-2\varepsilon_{i+1})-(q_{i+1}(1+\alpha_{i+1}-2\varepsilon_{i+1})-d-1)\varepsilon_{i+1}^3)}\\
			\leq&|U_{\omega}^{w\ast v'(l)}|^{\frac{q_{i+1}(1+\alpha_{i+1}-2\varepsilon_{i+1})-\varepsilon_{i+1}^2)}{1+\alpha_{i+1}+2\varepsilon_{i+1}}}\text{ see \eqref{Control Uwv'} in lemma~\ref{lemma:Control Uwv'} and\eqref{Control varepsilon 2}}\\
			\leq&|U_{\omega}^{w\ast v'(l)}|^{q_{i+1}(1-4\varepsilon_{i+1})-\varepsilon_{i+1}^2}. 
			\end{align*} 
			
			So that for any $U_{\omega}^{w\ast v'(l)}\in G_{\omega,i+1}$, we have 
			\begin{equation}\label{Control eta Uwv'}
			\eta_{\omega}^{i+1}(U_{\omega}^{w\ast v'(l)})\leq |U_{\omega}^{w\ast v'(l)}|^{q_{i+1}(1-4\varepsilon_{i+1})-\varepsilon_{i+1}^2}.
			\end{equation}
			\item[step 3] Now define $K_{\omega}=\cap_{i\in\mathbb{N}}\cup_{j\geq i}\cup_{U_{\omega}^{w}\in G_{\omega,j}}U_{\omega}^{w}=\cap_{i\in\mathbb{N}}\cup_{U_{\omega}^{w}\in G_{\omega,i}}U_{\omega}^{w}$. From the construction we can easily get $K_{\omega}\subset  W(\phi,\omega)$.  Noticing the definition of $\{\eta_{\omega}^{i}\}_{i\in\mathbb{N}}$, we can distribute a probability measure $\eta_{\omega}$ on the algebra generated by $\cup_{i\in\mathbb{N}}G_{\omega,i}$ such that for any $U_{\omega}^{w}\in G_{\omega,i}$, $\eta_{\omega}(U_{\omega}^{w})=\eta_{\omega}^{i}(U_{\omega}^{w})$
			
			We have the following properties:

			\begin{enumerate}[(i)]
				\item\label{i}
				if $U_{\omega}^{w\ast v'(l_1)}$ and $U_{\omega}^{w\ast v'(l_2)}$ are different elements that belong to $G(\omega,w)\subset G_{\omega,i+1}$, $B_{l_1}\cap B_{l_2}=\emptyset$ recall that 
				$B_{l_j}=B\big(y_{l_j},2|U_{\omega}^{w\ast v(l_j)}|\big)$ and $y_{l_j}\in U_{\omega}^{w\ast v(l_j)}$ for $j=1,2$.	So their distance is at least $\max_{j\in\{1,2\}}|U_{\omega}^{w\ast v(l_j)}|$.	
				\item\label{ii} For any $U_{\omega}^{w\ast v'(l)}$ in $G(\omega,w)\subset G_{\omega,i+1}$, we denote the corresponding $y$ by $y_l$, then 
				$y_l\in U_{\omega}^{w\ast v(l)}\cap E(\omega,w)\subset B_{l}\cap E(\omega,w)\neq \emptyset$.
				\item\label{iii} For any $U_{\omega}^{w\ast v'(l)}$ in $G(\omega,w)\subset G_{\omega,i+1}$,
				\begin{equation}\label{F control}
				\eta_{\omega}(U_{\omega}^{w\ast v'(l)})\leq |U_{\omega}^{w\ast v'(l)}|^{q_{i+1}(1-4\varepsilon_{i+1})-\varepsilon_{i+1}^2}.
				\end{equation}
				\item\label{iv} Any $U_{\omega}^{w\ast v'(l)}$ in $G(\omega,w)\subset G_{\omega,i+1}$ is contained in  an element $U_{\omega}^w \in G_{\omega,i}$ such that
				\begin{equation}\label{eta upper in i}
				\eta_{\omega}(U_{\omega}^{w\ast v'(l)})\leq 2Q(d)\eta_{\omega}^{i}(U_{\omega}^{w})\zeta_{\omega,w}(B(y_l,2|U_{\omega}^{w\ast v}|)),
				\end{equation}
				where $q_i\in \mathcal{Q}_i$ is such that $\mathcal{T}_{j_i}'(q_i)=d_i$.
			\end{enumerate}
			Because of the separation property~(\ref{i}), the probability measure $\eta_{\omega}$ can be extends to the $\sigma$-algebra generated by $\cup_{i\in\mathbb{N}}G_{\omega,i}$
			and it is easy to notice that $\eta_{\omega}(K_{\omega})=1$. The measure $\eta_{\omega}$ can be also extended to $U$ by setting, for any $B\in \mathcal{B}(U)$, $\eta_{\omega}(B):=\eta_{\omega}(B\cap K_{\omega})$.
			\item[step 4] 
			Now let us estimate the lower Hausdorff dimension of $\eta_{\omega}$ and then get the lower bound for the Hausdorff dimension $K_{\omega}$ . If $q_0=0$, there is nothing to prove. So we assume that $q_0>0$.
			
			Let us fix a ball $B$ which is a subset of $U$ with length smaller than that of every element in $G_{\omega,1}$, and assume that $B(x,r)\cap K_{\omega}\neq \emptyset$. Let $U_{\omega}^w$ be the element of largest diameter in $\bigcup_{i\geq 1} G_{\omega,i}$ such that $B$ intersects at least two elements of $G_{\omega,i+1}$ and is included in $U_{\omega}^w\in G_{\omega,i}$. We remark that this implies that $B$ does not intersect any other element of $G_{\omega,i}$ and as a consequence $\eta_{\omega}(B)\leq \eta_{\omega}(U_{\omega}^w)$.
			
			\medskip
			
			Let us distinguish three cases:

			$\bullet$ $|B|\geq |U_{\omega}^w|$: from \eqref{F control} in property (\ref{iii}) we have 
			\begin{equation}\label{control B>L}
			\eta_{\omega}(B)\leq \eta_{\omega}(U_{\omega}^w)\leq |U_{\omega}^w|^{q_{i}(1-4\varepsilon_{i})-\varepsilon_{i}^2}\leq |B|^{q_{i}(1-4\varepsilon_{i})-\varepsilon_{i}^2}.
			\end{equation}
			
			$\bullet$  $|B|\leq \frac{1}{3}|U_{\omega}^w|\exp(-(C+4)|w|\varepsilon_{i+1}^3)$. 
			Assume $U_{\omega}^{w\ast v'(1)},\dots,U_{\omega}^{w\ast v'(k)}$ are the elements of $G(\omega,w)\subset G_{\omega,i+1}$ which have non-empty intersection with $B$.
			From \eqref{eta upper in i} in property (\ref{iv}), we get
			\begin{equation*}
			\eta_{\omega}(B)=\sum_{l=1}^k \eta_{\omega}(B\cap U_{\omega}^{w\ast v'(l)})\leq 2 Q(d) \eta_{\omega}((U_{\omega}^w))\sum_{l=1}^k \zeta_{\omega,w}(U_{\omega}^{w\ast v(l)}).
			\end{equation*}
			
			From property (\ref{i}) we can also deduce that $\max\{|U_{\omega}^{w\ast v(l)}|:1\leq l\leq k\}\leq |B|$. From property (\ref{ii}) we can get $y_l\in U_{\omega}^{w\ast v(l)}\cap E(w)\neq\emptyset$. If $y\in\{y_l,1\leq l\leq k\}$, we have $B(y,3|B|)\supset(\bigcup_{l=1}^k U_{\omega}^{w\ast v(l)}) $.
			
			Using lemma \ref{lemma: Control zeta Bxr}, we can get:  
			$$\zeta_{\omega,w}(B(y,3|B|))\leq (2\cdot(3|B|)^{-\varepsilon_{i+1}^3})^d (\frac{3|B|}{|U_{\omega}^{w}|})^{q_{i+1}(1+\alpha_{i+1}-2\varepsilon_{i+1})}.$$
			So that 
			\begin{align*}
			\eta_{\omega}(B)&\leq 2 Q(d) \eta_{\omega}((U_{\omega}^w))\sum_{l=1}^k \zeta_{\omega,w}(U_{\omega}^{w\ast v(l)})\\
			&\leq 2 Q(d) \eta_{\omega}((U_{\omega}^w)) \zeta_{\omega,w}(B(y,3|B|))\\
			&\leq 2 Q(d) |U_{\omega}^w|^{q_{i}(1-4\varepsilon_{i})-\varepsilon_{i}^2} (2\cdot(3|B|)^{-\varepsilon_{i+1}^3})^d (\frac{3|B|}{|U_{\omega}^{w}|})^{q_{i+1}(1+\alpha_{i+1}-2\varepsilon_{i+1})}\\
			&\leq 2 Q(d)(2\cdot(3|B|)^{-\varepsilon_{i+1}^3})^d |U_{\omega}^w|^{q_{i}(1-4\varepsilon_{i})-\varepsilon_{i}^2}\\ &\cdot(\frac{3|B|}{|U_{\omega}^{w}|})^{q_{i}(1-4\varepsilon_{i})-\varepsilon_{i}^2+(q_{i+1}(1+\alpha_{i+1}-2\varepsilon_{i+1})-q_{i}(1-4\varepsilon_{i})+\varepsilon_{i}^2)}\\
			&\leq 2\cdot 2^d Q(d)(3|B|)^{q_{i}(1-4\varepsilon_{i})-\varepsilon_{i}^2-d\varepsilon_{i+1}^3} (\frac{3|B|}{|U_{\omega}^{w}|})^{(q_{i+1}(1+\alpha_{i+1}-2\varepsilon_{i+1})-q_{i}(1-4\varepsilon_{i}))}.
			\end{align*}
			From $\liminf_{i\to\infty}\alpha_{i}>0$ and $\lim_{i\to\infty} q_i=q_0>0$
			we have $q_{i+1}(1+\alpha_{i+1}-2\varepsilon_{i+1})-q_{i}(1-4\varepsilon_{i})\geq 0$  for $i$ large enough. Further more since $3|B|/|U_{\omega}^{w}|\leq 1$, we have:
			\begin{equation}\label{control B<<L}
			\eta_{\omega}(B)\leq 2^{d+1} Q(d)(3|B|)^{q_{i}(1-4\varepsilon_{i})-\varepsilon_{i}^2-d\varepsilon_{i+1}^3}. 
			\end{equation}

			$\bullet$  $\frac{1}{4}|U_{\omega}^{w}|\exp(-(C+4)|w|(\varepsilon_{i+1})^3)\leq |B|\leq |U_{\omega}^{w}|$:
			we need at most $M(B)$ balls $(B(k))_{1\leq k \leq M(B)}$ with diameter $\frac{1}{4}|U_{\omega}^{w}|\exp(-(C+4)|w|(\varepsilon_{i+1})^3)$ that can cover $B$ with $M(B)\leq \Gamma(d)(4\exp((C+4)|w|(\varepsilon_{i+1})^3))^d$, where $\Gamma$ is a function just depends on $d$ since we are dealing with problems in $\mathbb{R}^d$. For these Balls we have the estimate above. Consequently,
			
			\begin{equation*}\label{}
			\begin{split}
			\eta_{\omega}(B) \leq&\sum_{k=1}^{M(B)}\eta_{\omega}(B(k))\\ 
			\leq&\sum_{k=1}^{M(B)}2^{d+1} Q(d)(3|B(k)|)^{q_{i}(1-4\varepsilon_{i})-\varepsilon_{i}^2-d\varepsilon_{i+1}^3}\\
			\leq& \Gamma(d)(4\exp((C+4)|w|(\varepsilon_{i+1})^3))^d \cdot 2^{d+1} Q(d)(3|B|)^{q_{i}(1-4\varepsilon_{i})-\varepsilon_{i}^2-d\varepsilon_{i+1}^3}\\
			\leq&2^{3d+1}\cdot Q(d)\cdot\Gamma(d)(3|B|)^{q_{i}(1-4\varepsilon_{i})-\varepsilon_{i}^2-d\varepsilon_{i+1}^3}|U_{\omega}^w|^{-\frac{(C+4)\varepsilon_{i}^3}{c/2}}\ \text{see \eqref{Control U--upper--lower}}\\
			\leq& 2^{3d+1}\cdot Q(d)\cdot\Gamma(d)(3|B|)^{q_{i}(1-4\varepsilon_{i})-\varepsilon_{i}^2-d\varepsilon_{i+1}^3}|B|^{-\frac{(C+4)\varepsilon_{i}^3}{c/2}}
			\end{split}
			\end{equation*}
			Finally,			\begin{equation}\label{control B<L}
			\eta_{\omega}(B)\leq 2^{3d+1}\cdot Q(d)\cdot\Gamma(d)(3|B|)^{q_{i}(1-4\varepsilon_{i})-\varepsilon_{i}-d\varepsilon_{i+1}^3}|B|^{-\frac{2(C+4)\varepsilon_{i}^3}{c}}.
			\end{equation} 
			
			It follows from the estimations \eqref{control B>L},\eqref{control B<<L} and \eqref{control B<L} that 			$$\underline \dim_H (\eta_{\omega})\geq\lim_{i\to \infty}q_i=q_0,$$ then $\dim_{H}K_{\omega}\geq q_0$, so is $\dim_{H}W(\omega,\phi)\geq q_0$.

		\end{description}

	\end{proof}
	
	\section{Remarks of corollary~\ref{W'} and corollary~\ref{W''}}\label{section:corrollary}
	
	This section is mainly to explain the result in corollary~\ref{W'} and corollary~\ref{W''}. For the upper bound of the Hausdorff dimension, it is the same in section~\ref{section:Upper bound}, where we use a very natural cover to get the control of the upper bound.
	
	Now we begin to explain the lower bound. 
	
	{\bf For corollary~\ref{W'}.} The lower bound is almost the same of section~\ref{section:Lower bound} except the preparation of the choice of $v'(l)$ in step 2 .  
	
	The important procedure of choosing $p$ and $v'(l)$ in step 2 is to search a point $x\in X_{\omega}^{w\ast v(l)}$ so that to make sure the existence of them.
	
	Let us compare $z_{\sigma^n\omega}\in X_{\sigma^n\omega}$ and \eqref{Ass W'-1} in assumption of corollary~\ref{W'}.
	
	In step 2 of the proof of theorem~\ref{theorem lower bound}, for given $w$ and $v(l)$, \eqref{M big} implies for $1\leq s\leq l(\sigma^{m+\mathcal{N}+p_{i+1}+\mathcal{N}}\omega)$, we can joint the words $w\ast v(l)$ and $s$ as $w\ast v(l)\ast s$ and $z_{\sigma^{m+\mathcal{N}+n+\mathcal{N}}\omega}\in X_{\sigma^{m+\mathcal{N}+n+\mathcal{N}}\omega}$ ensures that we can find $s$ with $z_{\sigma^{m+\mathcal{N}+n+\mathcal{N}}\omega}\in X_{\sigma^{m+\mathcal{N}+n+\mathcal{N}}\omega}^s$. So we have 
	$$x=:g_{\omega}^{w\ast v(l)\ast}(z_{\sigma^{m+\mathcal{N}+n+\mathcal{N}}\omega})\in X_{\omega}^{w\ast v(l)\ast s}\subset X_{\omega}^{w\ast v(l)},$$
	and then the existence of $p$ and $v'(l)$.

	For corollary~\ref{W'}, we asked $\widetilde{M}$ large enough such that \eqref{M big} can be changed as 
	$$\mathbb{P}(\{\omega\in\Omega: M(\omega)\leq \widetilde{M},\ M'(\omega)\leq \widetilde{M}\})>7/8.$$
	Then continue the process in section~\ref{section:Basic properties}.  For given $w$ and $v(l)$, assumption \eqref{Ass W'-1} implies that there exists $\widetilde{v}\in \Sigma_{\sigma^{m+\mathcal{N}+n}\omega,k}$ with $k\leq\widetilde{M}\leq \mathcal{N}$ 
	such that $w\ast v(l)\widetilde{v}\in\Sigma_{\omega,m+\mathcal{N}+n+k}$ and $x=:g_{\omega}^{w\ast v(l)\widetilde{v}}(z_{\sigma^{m+\mathcal{N}+n+k}\omega}^{w\ast v(l)\widetilde{v}})\in X_{\omega}^{w\ast v(l)\widetilde{v}}\subset X_{\omega}^{w\ast v(l)}$. This will also implies the existence of $p$ and $v'(l)$. In fact the word $w\ast v(l)\widetilde{v}$ plays the similar role as $w\ast v(l)\ast s$ in the proof of theorem~\ref{theorem lower bound}.

	{\bf For corollary~\ref{W''}.}
	Define $z_{\sigma^{|v|}\omega}^{v}=:T_{\omega}^{v}x_{\omega}^v=x_{\omega}^v\in X_{\sigma^{|v|}\omega}$ and $x_{\omega}^v\in X_{\omega}$, we get $z_\omega^v\in  X_{\sigma^{|v|-1}\omega}$. \eqref{Ass W''-1} in the assumption implies \eqref{Ass W'-1} (if we do not distinguish $M'$ and $M''$), so for given $w$ and $v(l)$, we can also find $\widetilde{v}$ such that $w\ast v(l)\widetilde{v}\in\Sigma_{\omega,m+\mathcal{N}+n+k}$, and $z_{\sigma^{m+\mathcal{N}+n+k}\omega}^{w\ast v(l)\widetilde{v}}\in X_{\omega}^{w\ast v(l)\widetilde{v}}\subset X_{\omega}^{w\ast v(l)}$ with $k\leq \mathcal{N}$. So that we can also choose $p$ and $v'(l)$ as in step 2 in the proof of theorem~\ref{theorem lower bound}.
	
	Noticing $T_{\omega}^{w\ast v(l)\widetilde{v}} z_{\sigma^{m+\mathcal{N}+n+k}\omega}^{w\ast v(l)\widetilde{v}}=z_{\sigma^{m+\mathcal{N}+n+k}\omega}^{w\ast v(l)\widetilde{v}}\in X_{\sigma^{m+\mathcal{N}+n+k}\omega}\cap X_{\omega}^{w\ast v(l)\widetilde{v}} $,
	\begin{align*}
	\|T_{\omega}^{w\ast v(l)\widetilde{v}}x-x\|\leq& \|T_{\omega}^{w\ast v(l)\widetilde{v}}x-T_{\omega}^{w\ast v(l)\widetilde{v}} z_{\sigma^{m+\mathcal{N}+n+k}\omega}^{w\ast v(l)\widetilde{v}}\|\\
	&+\|T_{\omega}^{w\ast v(l)\widetilde{v}} z_{\sigma^{m+\mathcal{N}+n+k}\omega}^{w\ast v(l)\widetilde{v}}-z_{\sigma^{m+\mathcal{N}+n+k}\omega}^{w\ast v(l)\widetilde{v}}\|\\
	&+\|z_{\sigma^{m+\mathcal{N}+n+k}\omega}^{w\ast v(l)\widetilde{v}}-x\|\\
	=& \|T_{\omega}^{w\ast v(l)\widetilde{v}}x-T_{\omega}^{w\ast v(l)\widetilde{v}} z_{\sigma^{m+\mathcal{N}+n+k}\omega}^{w\ast v(l)\widetilde{v}}\|+\|z_{\sigma^{m+\mathcal{N}+n+k}\omega}^{w\ast v(l)\widetilde{v}}-x\|
	\end{align*} 
	Since $T_{\omega}^{w\ast v(l)\widetilde{v}}$ is essentially expanding, the term $\|z_{\sigma^{m+\mathcal{N}+n+k}\omega}^{w\ast v(l)\widetilde{v}}-x\|$ can be ignored with respect the term $\|T_{\omega}^{w\ast v(l)\widetilde{v}}x-T_{\omega}^{w\ast v(l)\widetilde{v}} z_{\sigma^{m+\mathcal{N}+n+k}\omega}^{w\ast v(l)\widetilde{v}}\|$, which means for any $x\in U_{\omega}^{w\ast v'(l)}\subset U_{\omega}^{w\ast v(l)\ast s}$, 
	we can also imply \eqref{Condition Tnx to z}.
	The control of the lower bound of the Hausdorff dimension for $W''(\phi,\omega)$ is almost the same as for $W'(\phi,\omega)$ if we choose proper $\{z_{\omega}^v: \omega\in\Omega,v\in\Sigma_{\omega,*}\}$.
	   
    
	\medskip
	Received xxxx 20xx; revised xxxx 20xx.
	\medskip
	
\end{document}